\documentclass[12pt,twoside]{article}
\usepackage[left=1.5cm,right=1.5cm,top=3cm,bottom=2cm]{geometry}
\usepackage{paralist}
\usepackage{marginnote}
\usepackage{amsfonts}
\usepackage{enumitem}
\usepackage{tikz}
\usetikzlibrary{matrix}
\usepackage{amsthm}
\usepackage{amsrefs}
\usepackage{amssymb}
\usepackage{amsmath}
\usepackage{graphicx}
\usepackage{varwidth}
\usepackage{hyperref}
\usepackage{fancyhdr}
\usepackage{stmaryrd}
\usepackage{multirow}

\usepackage{blkarray}
\usepackage[autostyle]{csquotes}
\usepackage[mathscr]{euscript}
\hypersetup{
    colorlinks=true,
    linkcolor=black,
    filecolor=magenta,      
    urlcolor=black,
    citecolor=black
}
\DeclareMathAlphabet{\mathcalligra}{T1}{calligra}{c}{h}
\providecommand{\U}[1]{\protect\rule{.1in}{.1in}}
\newtheorem{theorem}{Theorem}[section]
\newtheorem{proposition}[theorem]{Proposition}
\newtheorem{lemma}[theorem]{Lemma}

\newtheorem{corollary}[theorem]{Corollary}

\newtheorem{remark}[theorem]{Remark}
\let\oldremark\remark
\renewcommand{\remark}{\oldremark\normalfont}

\newtheorem{example}[theorem]{Example}
\let\oldexample\example
\renewcommand{\example}{\oldexample\normalfont}

\newtheorem{examples}[theorem]{Examples}
\let\oldexamples\examples
\renewcommand{\examples}{\oldexamples\normalfont}

\def\<{{\langle}}
\def\>{{\rangle}}

\def\bea{\begin{eqnarray*} }
\def\eea{\end{eqnarray*} }
\def\be{\begin{equation} }
\def\ee{\end{equation} }

\def\qed{\ifhmode\unskip\nobreak\fi\ifmmode\ifinner
\else\hskip5 pt \fi\fi\hbox{\hskip5 pt \vrule width4 pt
height6 pt  depth1.5 pt \hskip 1pt }}

\DeclareMathOperator*{\diver}{div}

\DeclareMathOperator*{\vol}{vol}

\DeclareMathOperator*{\supp}{supp}

\DeclareMathOperator*{\grad}{grad}
\DeclareMathOperator*{\ess}{ess}
\DeclareMathOperator*{\Lip}{Lip}
\DeclareMathOperator*{\av}{av}
\DeclareMathOperator*{\tr}{tr}
\DeclareMathOperator*{\ad}{ad}
\begin{document}
	
\title{Spectral estimates for Riemannian submersions with fibers of basic mean curvature}
\author{Panagiotis Polymerakis}
\date{}

\maketitle

\renewcommand{\thefootnote}{\fnsymbol{footnote}}
\footnotetext{\emph{Date:} \today} 
\renewcommand{\thefootnote}{\arabic{footnote}}

\renewcommand{\thefootnote}{\fnsymbol{footnote}}
\footnotetext{\emph{2010 Mathematics Subject Classification.} 58J50, 35P15, 53C99.}
\renewcommand{\thefootnote}{\arabic{footnote}}

\renewcommand{\thefootnote}{\fnsymbol{footnote}}
\footnotetext{\emph{Key words and phrases.} Riemannian submersion, basic mean curvature, Riemannian principal bundle, amenable Lie group, bottom of spectrum, discrete spectrum.}
\renewcommand{\thefootnote}{\arabic{footnote}}

\begin{abstract}
For Riemannian submersions with fibers of basic mean curvature, we compare the spectrum of the total space with the spectrum of a Schr\"{o}dinger operator on the base manifold. 
Exploiting this concept, we study submersions arising from actions of Lie groups. In this context, we extend the state of the art results on the bottom of the spectrum under Riemannian coverings. As an application, we compute the bottom of the spectrum and the Cheeger constant of connected, amenable Lie groups.
\end{abstract}

\section{Introduction}

The study of the spectrum of the Laplacian on a Riemannian manifold has attracted much attention over the last years. In order to comprehend its relations with the geometry of the underlying manifold, it is reasonable to investigate its behavior under maps between Riemannian manifolds that respect the geometry of the manifolds to some extent. In this paper, we study the behavior of the spectrum under Riemannian submersions.

The notion of Riemannian submersion was introduced in the 1960s as a tool to express the geometry of a manifold in terms of the geometry of simpler components, namely, the base space and the fibers. Of course, by geometry of the fibers we mean both their intrinsic and their extrinsic geometry as submanifolds of the total space. Bearing this in mind, it is natural to describe the spectrum of the total space in terms of the geometry and the spectrum of the base space and the fibers.

To set the stage, let $p \colon M_{2} \to M_{1}$ be a Riemannian submersion and denote by $F_{x} := p^{-1}(x)$ the fiber over $x \in M_{1}$. The spectrum of (the Laplacian on) $M_{2}$ has been studied in the case where $M_{2}$ is closed (that is, compact and without boundary) and the submersion has totally geodesic, or minimal fibers, or fibers of basic mean curvature (cf. for example the survey \cite{MR2963622}). However, the situation is quite more complicated and yet unclear if $M_{2}$ is not closed.

Recently, in \cite{Mine4}, extending the result of \cite{MR3787357}, we established a lower bound for the bottom of the spectrum $\lambda_{0}(M_{2})$ of $M_{2}$, if the mean curvature of the fibers is bounded in a certain way. More precisely, according to \cite[Theorem 1.1]{Mine4}, if the (unnormalized) mean curvature of the fibers is bounded by $\| H \| \leq C \leq 2 \sqrt{\lambda_{0}(M_{1})}$, then the bottom of the spectrum of $M_{2}$ satisfies
\[
\lambda_{0}(M_{2}) \geq (\sqrt{\lambda_{0}(M_{1})} - C/2)^{2} + \inf_{x \in M_{1}} \lambda_{0}(F_{x}).
\]
Moreover, if the equality holds and $\lambda_{0}(M_{1}) \notin \sigma_{\ess}(M_{1})$ (that is, $\lambda_{0}(M_{1})$ is an isolated point of the spectrum of the Laplacian on $M_{1}$), then $\lambda_{0}(F_{x})$ is equal to its infimum for almost any $x \in M_{1}$. Recall that, in general, $\lambda_{0}(F_{x})$ is only upper semi-continuous with respect to $x \in M_{1}$ (cf. \cite[Lemma 2.9]{Mine4}).

In the second part of \cite{Mine4}, following \cite{Bessa}, we studied Riemannian submersions with closed fibers. In this context, we introduced a Schr\"{o}dinger operator on $M_{1}$, whose potential is determined by the volume of the fibers, and compared its spectrum with the spectrum of $M_{2}$. It should be noticed that if the submersion has fibers of infinite volume, then we are not able to define that operator, at least in the way we did in \cite{Mine4}.

In this paper, motivated by the aforementioned results, we introduce a Schr\"{o}dinger operator on the base space of a Riemannian submersion with fibers of basic mean curvature and compare its spectrum with the spectrum of the total space. To be more specific, let $p \colon M_{2} \to M_{1}$ be a Riemannian submersion with fibers of basic mean curvature, and consider the Schr\"{o}dinger operator
\begin{equation}\label{operator}
S = \Delta + \frac{1}{4} \| p_{*} H \|^{2} - \frac{1}{2} \diver p_{*}H
\end{equation}
on $M_{1}$. It is worth to point out that $S$ is non-negative definite, that is, $\lambda_{0}(S) \geq 0$. Furthermore, it is evident that $S$ coincides with the Laplacian, if the submersion has minimal fibers. Our first result relates the bottom of the spectrum of this operator with the bottom of the spectrum of $M_{2}$.

\begin{theorem}\label{basic mean curv thm}
Let $p \colon M_{2} \to M_{1}$ be a Riemannian submersion with fibers of basic mean curvature, and consider the Schr\"{o}dinger operator $S$ as above. Then
\[
\lambda_{0}(M_{2}) \geq \lambda_{0}(S) + \inf_{x \in M_{1}} \lambda_{0}(F_{x}).
\]
If, in addition, the equality holds and $\lambda_{0}(S) \notin \sigma_{\ess}(S)$, then $\lambda_{0}(F_{x})$ is almost everywhere equal to its infimum.
\end{theorem}

It should be emphasized that no assumptions on the geometry or the topology of the manifolds are required in this theorem. In particular, the manifolds do not have to be complete. This, together with the decomposition principle, allows us to derive a similar inequality involving the bottoms of the essential spectra, if the fibers are closed.

It is worth to mention that in some cases $\lambda_{0}(S)$ can be estimated in terms of $\lambda_{0}(M_{1})$. For instance, if the mean curvature of the fibers is bounded by $\| H \| \leq C \leq 2 \sqrt{\lambda_{0}(M_{1})}$, then the bottoms of the spectra satisfy
\[
\lambda_{0}(S) \geq (\sqrt{\lambda_{0}(M_{1})} - C/2)^{2}.
\]
Thus, Theorem \ref{basic mean curv thm} provides a sharper lower bound for $\lambda_{0}(M_{2})$ than \cite[Theorem 1.1]{Mine4} in the case where both of them are applicable.

It is noteworthy that if the submersion has closed fibers, then the operator $S$ defined in (\ref{operator}) coincides with the Schr\"{o}dinger operator introduced in \cite{Mine4}, and there is a remarkable relation with the work of Bordoni \cite{Bordoni} on Riemannian submersions with fibers of basic mean curvature. Given such a submersion $p \colon M_{2} \to M_{1}$ with $M_{2}$ closed, Bordoni considered the restrictions $\Delta_{c}$ and $\Delta_{0}$ of the Laplacian acting on lifted functions and on functions whose average is zero on any fiber, respectively, and showed in \cite[Theorem 1.6]{Bordoni} that the spectrum is written as $\sigma(M_{2}) = \sigma(\Delta_{c}) \cup \sigma(\Delta_{0})$. In this setting, the spectrum of the operator $S$ coincides with the spectrum of $\Delta_{c}$. It should be observed that expressing the latter one as the spectrum of an operator on $M_{1}$ allows us to relate it more easily to the spectrum of $M_{1}$. For Riemannian submersions with closed fibers we obtain the following consequence of Theorem \ref{basic mean curv thm} (compare with \cite[Theorem 1.2]{Mine4}), where we denote by $\lambda_{0}^{\ess}$ the bottom of the essential spectrum of an operator.

\begin{corollary}\label{closed basic}
If $p \colon M_{2} \to M_{1}$ is a Riemannian submersion with closed fibers of basic mean curvature, then $\lambda_{0}(M_{2}) = \lambda_{0}(S)$ and $\lambda_{0}^{\ess}(M_{2}) = \lambda_{0}^{\ess}(S)$. In particular, $M_{2}$ has discrete spectrum if and only if the spectrum of $S$ is discrete.
\end{corollary}

This corollary generalizes \cite[Theorem 1(ii)]{Bessa}, which asserts that if $p \colon M_{2} \to M_{1}$ is a Riemannian submersion with closed and minimal fibers, then $M_{1}$ has discrete spectrum if and only if $M_{2}$ has discrete spectrum. This equivalence has been extended in \cite[Corollary 1.4]{Mine4} under the weaker assumption that the fibers are closed and of bounded mean curvature. Corollary \ref{closed basic} characterizes the discreteness of the spectrum of $M_{2}$ in terms of $S$ instead of the Laplacian, which, nonetheless, is very natural. More precisely, for warped products of the form $M \times_{\psi} F$ with $F$ closed, this characterization coincides with \cite[Theorem 3.3]{MR577877} of Baider.

If, in addition, the manifolds involved in Corollary \ref{closed basic} are complete, then we know from \cite[Theorem 1.2]{Mine4} that the spectra and the essential spectra satisfy $\sigma(S) \subset \sigma(M_{2})$ and $\sigma_{\ess}(S) \subset \sigma_{\ess}(M_{2})$. This, together with Theorem \ref{basic mean curv thm} and Corollary \ref{closed basic}, shows that it is very reasonable to compare the spectrum of $S$ with the spectrum of $M_{2}$, if the submersion has fibers of basic mean curvature.

In the second part of the paper, we study Riemannian principal bundles. To be more specific, let $G$ be a possibly discrete Lie group acting smoothly, freely and properly on a Riemannian manifold $M_{2}$ via isometries, where $\dim G < \dim M_{2}$. Such an action induces on $M_{1} = M_{2}/G$ the structure of Riemannian manifold. If $G$ is non-discrete, the projection $p \colon M_{2} \to M_{1}$ is a Riemannian submersion with fibers of basic mean curvature. We then say that $p$ is a \textit{Riemannian submersion arising from the action of} $G$.

In the case where $G$ is a discrete group, its action gives rise to a normal Riemannian covering. In this context, there are various results establishing relations between properties of the deck transformation group and the behavior of the spectrum. To be more precise, let $q \colon M_{2} \to M_{1}$ be a normal Riemannian covering with deck transformation group $\Gamma$. Then the bottoms of the spectra satisfy $\lambda_{0}(M_{2}) \geq \lambda_{0}(M_{1})$ (cf. for instance \cite{BMP1} and the references therein). Brooks was the first one to investigate when the equality holds. In \cite{Brooks}, he showed that if $M_{1}$ is closed, then $\Gamma$ is amenable if and only if $\lambda_{0}(M_{2}) = 0$. It is apparent that in this setting, we also have that $\lambda_{0}(M_{1}) = 0$. In \cite{BMP1}, we proved that if $\Gamma$ is amenable, then $\lambda_{0}(M_{2}) = \lambda_{0}(M_{1})$, without any assumptions on the topology or the geometry of $M_{1}$. It was established in \cite{Mine} that if, in addition, $M_{1}$ is complete, then $\sigma(M_{1}) \subset \sigma(M_{2})$. Conversely, according to \cite{Mine2}, if $\lambda_{0}(M_{2}) = \lambda_{0}(M_{1})$ and $\lambda_{0}(M_{1}) \notin \sigma_{\ess}(M_{1})$, then $\Gamma$ is amenable. 

If $G$ is non-discrete, then from the above discussion, it makes sense to compare the spectrum of the Laplacian on the total space with the spectrum of the Schr\"{o}dinger operator $S$ on the base manifold, defined in (\ref{operator}). It should be noticed that Theorem \ref{basic mean curv thm} implies that $\lambda_{0}(M_{2}) \geq \lambda_{0}(S)$. In the following theorem, we extend the aforementioned results to Riemannian submersions arising from Lie group actions, where we denote by $G_{0}$ the connected component of $G$.

\begin{theorem}\label{submersion group thm}
Let $p \colon M_{2} \to M_{1}$ be a Riemannian submersion arising from the action of a Lie group $G$. Then:
\begin{enumerate}[topsep=0pt,itemsep=-1pt,partopsep=1ex,parsep=0.5ex,leftmargin=*, label=(\roman*), align=left, labelsep=0em]
\item If $G$ is amenable and $G_{0}$ is unimodular, then $\lambda_{0}(M_{2}) = \lambda_{0}(S)$.
\item If, in addition, $M_{1}$ is complete, then $\sigma(S) \subset \sigma(M_{2})$.
\item Conversely, if $\lambda_{0}(M_{2}) = \lambda_{0}(S)$ and $\lambda_{0}(S) \notin \sigma_{\ess}(S)$, then $G$ is amenable and $G_{0}$ is unimodular.
\end{enumerate}
\end{theorem}

Recall that there exist connected Lie groups that are amenable but not unimodular (because any solvable group is amenable), and connected Lie groups that are unimodular but not amenable (since any connected, semisimple Lie group is unimodular).

It is notable that if $G$ is compact, then Corollary \ref{closed basic} compares the spectra and the essential spectra of the operators. Even though Theorem \ref{submersion group thm} is formulated in terms of spectra, it also provides information about the essential spectra. This follows from the fact that if $G$ is non-compact, then $\sigma(M_{2}) = \sigma_{\ess}(M_{2})$ (cf. for example \cite[Theorem 5.2]{Mine}).

As in the context of Riemannian coverings, it is plausible to wonder if the assumption $\lambda_{0}(S) \notin \sigma_{\ess}(S)$ can be weakened in Theorem \ref{submersion group thm}(iii). We will construct a wide class of examples demonstrating that this assumption is essential. Namely, let $M$ be any Riemannian manifold with $\lambda_{0}(M) \in \sigma_{\ess}(M)$. We will show that there exists a Riemannian submersion $p \colon M_{2} \to M_{1} := M$ with minimal fibers, arising from the action of a connected, non-unimodular Lie group $G$, such that $\lambda_{0}(M_{2}) = \lambda_{0}(M_{1})$. Since the submersion has minimal fibers, it is clear that $S$ coincides with the Laplacian on $M_{1}$.

In the case where the base manifold is closed, we derive another analogue of Brooks' result, which is slightly different. This is because in Theorem \ref{submersion group thm} we investigate the validity of $\lambda_{0}(M_{2}) = \lambda_{0}(S)$, while the following corollary characterizes the stronger property $\lambda_{0}(M_{2}) = 0$.

\begin{corollary}\label{closed cor}
Let $p \colon M_{2} \to M_{1}$ be a Riemannian submersion arising from the action of a Lie group $G$, where $M_{1}$ is closed. Then $G$ is unimodular and amenable if and only if $\lambda_{0}(M_{2}) = 0$.
\end{corollary}

Finally, exploiting Theorems \ref{basic mean curv thm} and \ref{submersion group thm}, we study quotients of Lie groups by normal subgroups. In this setting, we obtain some relations between the mean curvature of the subgroup and the bottom of the spectrum of the group, the subgroup and the quotient.

\begin{theorem}\label{Lie group thm}
Let $G$ be a connected Lie group endowed with a left-invariant metric and $N$ be a closed (as a subset), connected, normal subgroup of $G$ with mean curvature $H$. Then
\[
\lambda_{0}(G) \geq \lambda_{0}(G/N) + \lambda_{0}(N) -\frac{1}{4} \| H \|^{2} + \frac{1}{2} \tr (\ad H).
\]
Moreover, $N$ is unimodular and amenable if and only if
\[
\lambda_{0}(G) = \lambda_{0}(G/N) -\frac{1}{4} \| H \|^{2} + \frac{1}{2} \tr (\ad H).
\]
\end{theorem}

As an application of this theorem, we compute the bottom of the spectrum and the Cheeger constant of connected, amenable Lie groups. This extends the result of \cite{MR2053361} in various ways.

\begin{corollary}\label{bottom amen}
Let $G$ be a connected, amenable Lie group endowed with a left-invariant metric. Then the bottom of its spectrum and its Cheeger constant are given by
\[
\lambda_{0}(G) = \frac{1}{4} h(G)^{2} = \frac{1}{4} \max_{X \in \mathfrak{g}, \| X \| =1} (\tr (\ad X))^{2}. 
\]
If $G$ is not unimodular, then the maximum is achieved by the unit vector in the direction of the mean curvature (in $G$) of the commutator subgroup $[S,S]$ of the radical $S$ of $G$.
\end{corollary}

The paper is organized as follows: In Section \ref{preliminaries}, we discuss some basic properties of Schr\"{o}dinger operators, Riemannian submersions and Lie groups. In particular, we provide a spectral theoretic characterization for connected, amenable and unimodular Lie groups, which is well known for simply connected Lie groups. In Section \ref{sec basic mean}, we study Riemannian submersions with fibers of basic mean curvature and prove Theorem \ref{basic mean curv thm} and Corollary \ref{closed basic}. In Section \ref{sec principal}, we focus on submersions arising from Lie group actions and establish Theorem \ref{submersion group thm} and Corollary \ref{closed cor}. In Section \ref{Lie group sec}, we discuss some consequences of our results to Lie groups and show Theorem \ref{Lie group thm} and Corollary \ref{bottom amen}.

\medskip

\textbf{Acknowledgements.} I would like to thank Werner Ballmann and Dorothee Sch\"{u}th for their helpful comments and remarks. I am also grateful to the Max Planck Institute for Mathematics in Bonn for its support and hospitality.

\section{Preliminaries}\label{preliminaries}

Throughout this paper manifolds are assumed to be connected and without boundary, unless otherwise stated, except for Lie groups. Let $M$ be a possibly non-connected Riemannian manifold. A \textit{Schr\"{o}dinger operator} on $M$ is an operator of the form $S = \Delta + V$, where $\Delta$ is the Laplacian on $M$ and $V \in C^{\infty}(M)$, such that there exists $c \in \mathbb{R}$ satisfying
\[
\langle S f, f \rangle_{L^{2}(M)} \geq c \| f \|_{L^{2}(M)}^{2}
\]
for any $f \in C^{\infty}_{c}(M)$. Then the operator
\[
S \colon C^{\infty}_{c}(M) \subset L^{2}(M) \to L^{2}(M)
\]
admits Friedrichs extension, being densely defined, symmetric and bounded from below. It is worth to point out that if $M$ is complete, then this operator is essentially self-adjoint; that is, its closure coincides with its Friedrichs extension (cf. \cite[Proposition 2.4]{Mine4}).

The spectrum and the essential spectrum of (the Friedrichs extension of) $S$ are denoted by $\sigma(S)$ and $\sigma_{\ess}(S)$, respectively, and their bottoms (that is, their minimums) by $\lambda_{0}(S)$ and $\lambda_{0}^{\ess}(S)$, respectively. In the case of the Laplacian (that is, $V=0$) we write $\sigma(M), \sigma_{\ess}(M)$ and $\lambda_{0}(M), \lambda_{0}^{\ess}(M)$ for these sets and quantities. We have by definition that $\lambda_{0}^{\ess}(S) = + \infty$ if $S$ has \textit{discrete spectrum}, which means that $\sigma_{\ess}(S)$ is empty. If $\sigma_{\ess}(M)$ is empty, we say that $M$ has discrete spectrum. 

The \textit{Rayleigh quotient} of a non-zero $f \in \Lip_{c}(M)$ with respect to $S$ is defined by
\[
\mathcal{R}_{S}(f) := \frac{\int_{M} (\| \grad f \|^{2} + Vf^{2})}{\int_{M} f^{2}}.
\]
The Rayleigh quotient of $f$ with respect to the Laplacian is denoted by $\mathcal{R}(f)$, or by $\mathcal{R}_{g}(f)$ if the Riemannian metric $g$ of $M$ is not clear from the context. According to the next proposition, the bottom of the spectrum of $S$ can be expressed as an infimum of Rayleigh quotients (cf. for example \cite[Section 2]{Mine2} and the references therein).

\begin{proposition}\label{bottom}
Let $S$ be a Schr\"{o}dinger operator on a Riemannian manifold $M$. Then the bottom of the spectrum of $S$ satisfies
\[
\lambda_{0}(S) = \inf_{f} \mathcal{R}_{S}(f),
\]
where the infimum is taken aver all $f \in C^{\infty}_{c}(M) \smallsetminus \{0\}$, or over all $f \in \Lip_{c}(M) \smallsetminus \{0\}$.
\end{proposition}

A remarkable property of the essential spectrum of $S$ follows from the decomposition principle, which states that
\[
\sigma_{\ess}(S) = \sigma_{\ess}(S, M \smallsetminus K)
\]
for any smoothly bounded, compact domain $K$ of $M$. This is well known in the case where $M$ is complete (compare with \cite[Proposition 2.1]{MR544241}), but also holds if $M$ is non-complete (cf. for instance \cite[Theorem A.17]{BP}). The next proposition summarizes the properties of the bottom of the essential spectrum that will be used in the sequel.

\begin{proposition}\label{bottom essential spectrum}
Let $S$ be a Schr\"{o}dinger operator on a Riemannian manifold $M$ and consider an exhausting sequence $(K_{n})_{n \in \mathbb{N}}$ of $M$ consisting of smoothly bounded, compact domains. Then the bottom of the essential spectrum of $S$ is given by
\[
\lambda_{0}^{\ess}(S) = \lim_{n} \lambda_{0}(S, M \smallsetminus K_{n}).
\]
In particular, there exists $(f_{n})_{n \in \mathbb{N}} \subset C^{\infty}_{c}(M) \smallsetminus \{0\}$ with $\supp f_{n}$ pairwise disjoint, such that $\mathcal{R}_{S}(f_{n}) \rightarrow \lambda_{0}^{\ess}(S)$.
Furthermore, for any sequence $(f_{n})_{n \in \mathbb{N}} \subset C^{\infty}_{c}(M) \smallsetminus \{0\}$ with $\supp f_{n}$ pairwise disjoint, we have that
\[
\lambda_{0}^{\ess}(S) \leq \liminf_{n} \mathcal{R}_{S}(f_{n}).
\]
\end{proposition}

\begin{proof}
The third assertion may be found for example in \cite[Proposition 2.2]{Mine4}. From this and Proposition \ref{bottom}, it is not hard to see that
\[
\lambda_{0}^{\ess}(S) \leq \lim_{n} \lambda_{0}(S, M \smallsetminus K_{n}),
\]
while the decomposition principle gives that $\lambda_{0}^{\ess}(S) = \lambda_{0}^{\ess}(S, M \smallsetminus K_{n}) \geq \lambda_{0}(S, M \smallsetminus K_{n})$ for any $n \in \mathbb{N}$, as we wished. The proof is completed by the first part and Proposition \ref{bottom}.\qed
\end{proof}\medskip

For $\lambda \in \mathbb{R}$, a sequence $(f_{n})_{n \in \mathbb{N}} \subset C^{\infty}_{c}(M) \smallsetminus \{0\}$ is called a \textit{characteristic sequence} for $S$ and $\lambda$ if
\[
\frac{\| (S - \lambda) f_{n} \|_{L^{2}(M)}}{\| f_{n} \|_{L^{2}(M)}} \rightarrow 0, \text{ as } n \rightarrow + \infty.
\]
In general, the spectrum of a self-adjoint operator consists of approximate eigenvalues of the operator. In our context, if $M$ is complete, then $S$ is essentially self-adjoint. This allows us to characterize the spectrum of $S$ in terms of compactly supported smooth functions as follows.

\begin{proposition}\label{characteristic seq}
Let $S$ be a Schr\"{o}dinger operator on a complete Riemannian manifold $M$ and consider $\lambda \in \mathbb{R}$. Then $\lambda \in \sigma(S)$ if and only if there is a characteristic sequence for $S$ and $\lambda$.
\end{proposition}

Assume now that $\varphi \in C^{\infty}(M)$ is a positive solution of $S \varphi = \lambda \varphi$ for some $\lambda \in \mathbb{R}$. Denote by $L^{2}_{\varphi}(M)$ the $L^{2}$-space of $M$ with respect to the measure $\varphi^{2} d \vol$, where $d \vol$ is the volume element of $M$ induced by its Riemannian metric. It straightforward to verify that the isometric isomorphism $m_{\varphi} \colon L^{2}_{\varphi}(M) \to L^{2}(M)$, defined by $m_{\varphi} f = \varphi f$, intertwines $S$ with the diffusion operator
\[
L := m_{\varphi}^{-1} \circ (S - \lambda) \circ m_{\varphi} = \Delta - 2 \grad \ln \varphi.
\]
The operator $L$ is called \textit{renormalization} of $S$ with respect to $\varphi$. The Rayleigh quotient of a non-zero $f \in C^{\infty}_{c}(M)$ with respect to $L$ is defined by
\[
\mathcal{R}_{L}(f) := \frac{\langle Lf , f \rangle_{L^{2}_{\varphi}(M)}}{\| f \|_{L^{2}_{\varphi}(M)}^{2}} = \frac{\int_{M} \| \grad f\|^{2} \varphi^{2}}{\int_{M} f^{2} \varphi^{2}}.
\]
\begin{lemma}\label{renormalized comp}
For any non-zero $f \in C^{\infty}_{c}(M)$ and $c \in \mathbb{R}$, we have that:
\begin{enumerate}[topsep=0pt,itemsep=-1pt,partopsep=1ex,parsep=0.5ex,leftmargin=*, label=(\roman*), align=left, labelsep=0em]
\item $\mathcal{R}_{L}(f) = \mathcal{R}_{S}(\varphi f) - \lambda$,
\item $\| (L - c) f \|_{L^{2}_{\varphi}(M)} = \| (S - \lambda - c) (\varphi f)\|_{L^{2}(M)}$.
\end{enumerate}
\end{lemma}

\begin{proof}
Follows immediately from the definition of $L$ and the fact that $m_{\varphi}$ is an isometric isomorphism. \qed
\end{proof}

\subsection{Riemannian submersions}

Let $M_{1}$ and $M_{2}$ be Riemannian manifolds with $\dim M_{1} < \dim M_{2}$. A surjective smooth map $p \colon M_{2} \to M_{1}$ is called a \textit{submersion} if its differential is surjective at any point. The kernel of $p_{*y}$ is called the \textit{vertical space} at $y \in M_{2}$, and its orthogonal complement in $T_{y}M_{2}$ is called the \textit{horizontal space} at $y$. These spaces are denoted by $(T_{y}M_{2})^{v}$ and $(T_{y}M_{2})^{h}$, respectively. It is evident that the fiber $F_{x} := p^{-1}(x)$ over $x \in M_{1}$ is a possibly non-connected submanifold of $M_{2}$ and $(T_{y}M_{2})^{v}$ is the tangent space of $F_{x}$ at $y \in F_{x}$. The submersion $p$ is called \textit{Riemannian submersion} if the restriction $p_{*y} \colon (T_{y}M_{2})^{h} \to T_{p(y)}M_{1}$ is an isometry for any $y \in M_{2}$. For more details, see \cite{MR2110043}.

Given a Riemannian submersion $p \colon M_{2} \to M_{1}$, a smooth map $s \colon U \to M_{2}$ defined on an open subset $U$ of $M_{1}$, is called \textit{section} if $(p \circ s)(x) = x$ for any $x \in U$. We say that a section $s \colon U \subset M_{1} \to M_{2}$ is \textit{extensible} if it can be extended to a section $s^{\prime} \colon U^{\prime} \subset M_{1} \to M_{2}$ with $\bar{U} \subset U^{\prime}$.

A vector field $Y$ on $M_{2}$ is called \textit{horizontal} (\textit{vertical}) if $Y(y)$ belongs to the horizontal (vertical, respectively) space at $y$ for any $y \in M_{2}$. It is easily checked that any vector field $Y$ on $M_{2}$ can be written as $Y = Y^{h} + Y^{v}$ with $Y^{h}$ horizontal and $Y^{v}$ vertical. Moreover, any vector field $X$ on $M_{1}$ has a unique horizontal lift $\tilde{X}$ on $M_{2}$; that is, $\tilde{X}$ is horizontal and $p_{*}\tilde{X} = X$. A vector field $Y$ on $M_{2}$ is called \textit{basic} if $Y = \tilde{X}$ for some vector field $X$ on $M_{1}$.

The (unnormalized) \textit{mean curvature} of the fibers is defined by
\[
H(y) := \sum_{i=1}^{k} \alpha(X_{i},X_{i}),
\]
where $\alpha(\cdot, \cdot)$ is the second fundamental form of the fiber $F_{p(y)}$ and $\{X_{i}\}_{i=1}^{k}$ is an orthonormal basis of $(T_{y}M_{2})^{v}$. We say that the Riemannian submersion $p$ has \textit{minimal fibers} or \textit{fibers of basic mean curvature} if $H=0$ or $H$ is basic, respectively.

The \textit{lift} of a function $f \in C^{\infty}(M_{1})$ on $M_{2}$ is the smooth function $\tilde{f} := f \circ p$. The next lemma provides a simple expression for the Laplacian and the gradient of a lifted function.
\begin{lemma}\label{lift}
For any $f \in C^{\infty}(M_{1})$ and its lift $\tilde{f}$ on $M_{2}$, we have that:
\begin{enumerate}[topsep=0pt,itemsep=-1pt,partopsep=1ex,parsep=0.5ex,leftmargin=*, label=(\roman*), align=left, labelsep=0em]
\item $\grad \tilde{f} = \widetilde{\grad f}$,
\item $ \Delta \tilde{f} = \widetilde{\Delta f} + \langle\widetilde{\grad f} , H \rangle$.
\end{enumerate}
\end{lemma}

\begin{proof}
Both statements follow from elementary computations, which may be found for example in \cite[Subsection 2.2]{Bessa}. \qed
\end{proof}

\subsection{Lie groups}

In this subsection, we recall some basic definitions and results about Lie groups, and discuss some consequences of the Cheeger and Buser inequalities in this setting. 

For a Borel subset $A$ of a Riemannian manifold $(M,g)$ we denote the volume of $A$ by $|A|_{g}$, or simply by $|A|$ if the Riemannian metric of $M$ is clear from the context. Similarly, for an $m$-dimensional submanifold $N$ of $M$, we denote by $|N|$ the $m$-dimensional volume of $N$. The \textit{Cheeger constant} of a Riemannian manifold $M$ is defined by
\[
h(M) := \inf_{K} \frac{|\partial K|}{|K|},
\]
where the infimum is taken over all smoothly bounded, compact domains $K$ of $M$. It is related to the bottom of the spectrum via the Cheeger inequality (cf. \cite{Cheeger})
\[
\lambda_{0}(M) \geq \frac{h(M)^{2}}{4}.
\]
Buser \cite{Buser} established an inverse inequality for complete manifolds with Ricci curvature bounded from below. In particular, if $M$ is such a manifold, then $\lambda_{0}(M) = 0$ if and only if $h(M) = 0$. For our purposes, we also need the following lemma from his work, where $A^{r}$ stands for the $r$-tubular neighborhood of a subset $A$ of $M$.

\begin{lemma}[{Compare with \cite[Lemma 7.2]{Buser}; see also \cite[Corollary 6.3]{Mine}}]
Let $M$ be a non-compact, complete Riemannian manifold with Ricci curvature bounded from below. If $h(M) = 0$, then for any $\varepsilon, r > 0$, there exists an open, bounded $W \subset M$ such that
\[
| (\partial W)^{r}| < \varepsilon | W \smallsetminus (\partial W)^{r} |.
\]
\end{lemma}

Throughout this paper Lie groups are assumed to be non-discrete and possibly non-connected, unless otherwise stated.
A possibly discrete Lie group $G$ is called \textit{amenable} if there exists a left-invariant mean on $L^{\infty}(G)$; that is, a linear functional $\mu \colon L^{\infty}(G) \to \mathbb{R}$ such that
\[
\ess\inf f \leq \mu(f) \leq \ess \sup f \text{ and } \mu(f \circ L_{x}) = \mu(f),
\]
for any $f \in L^{\infty}(G)$ and $x\in G$, where $L_{x}$ stands for multiplication from the left with an element $x \in G$. Here, $L^{\infty}(G)$ is considered with respect to the Haar measure. If $G$ is non-discrete, then its Haar measure is a constant multiple of the volume element of $G$ induced from a left-invariant metric. If $G$ is discrete, then its Haar measure is a constant multiple of the counting measure. For more details, see \cite{MR0251549}.

\begin{lemma}\label{short exact}
If $N$ is a normal subgroup of a possibly discrete Lie group $G$, then $G$ is amenable if and only if $N$ and $G/N$ are amenable.
\end{lemma}

It is not hard to verify that abelian and compact Lie groups are amenable. Therefore, so are compact extensions of solvable groups. As a matter of fact, a connected Lie group is amenable if and only if it is a compact extension of a solvable group (cf. for example \cite[Lemma 2.2]{MR454886}).

Let $G$ be a connected Lie group with Lie algebra $\mathfrak{g}$. The \textit{radical} $\mathfrak{s}$ of $\mathfrak{g}$ is the largest solvable ideal of $\mathfrak{g}$. The \textit{radical} $S$ of $G$ is the connected subgroup with Lie algebra $\mathfrak{s}$. Then $S$ is a closed, normal subgroup of $G$ and the quotient $G/S$ is semisimple. In this case, we have that $G$ is amenable if and only if $G/S$ is compact (cf. \cite[p. 724f]{Hoke} and the references therein).

A Lie group is called \textit{unimodular} if its Haar measure is also right-invariant. For a connected Lie group, this property may be reformulated in terms of its Lie algebra as follows.

\begin{lemma}[{\cite[Proposition 1.2]{Hoke}}]\label{unimodular}
A connected Lie group $G$ is unimodular if and only if $\tr (\ad X) = 0$ for any $X$ in the Lie algebra of $G$.
\end{lemma}

It is worth to point out that connected, nilpotent Lie groups are unimodular and amenable. In addition, compact extensions of connected, unimodular Lie groups are unimodular (cf. \cite[Proposition 8]{MR330345}). 

Although the aforementioned properties are group theoretic, they are characterized by the spectrum of the Lie group according to the next theorem, which is well known for simply connected Lie groups.

\begin{theorem}\label{unim and amen}
A connected Lie group $G$ is unimodular and amenable if and only if $\lambda_{0}(G) = 0$ for some/any left-invariant metric on $G$.
\end{theorem}

\begin{proof}
We know from \cite[Theorem 3.8]{Hoke} that a simply connected Lie group $\tilde{G}$ is unimodular and amenable if and only if $h(\tilde{G}) = 0$ with respect to some/any left-invariant metric. By the Cheeger and Buser inequalities, this gives the assertion for simply connected Lie groups. To show its validity for a connected Lie group $G$, let $q \colon \tilde{G} \to G$ be the universal covering of $G$. It follows from Lemma \ref{unimodular} that $\tilde{G}$ is unimodular if and only if $G$ is unimodular, since their Lie algebras are isomorphic. Furthermore, $\pi_{1}(G)$ is abelian and isomorphic to the kernel of $q$ as a Lie group homomorphism. Therefore, $\tilde{G}$ is an extension of $G$ by an amenable group, and Lemma \ref{short exact} yields that $\tilde{G}$ is amenable if and only if $G$ is amenable. Taking into account that $\pi_{1}(G)$ is amenable, we conclude from \cite[Theorem 1.2]{BMP1} that $\lambda_{0}(\tilde{G}) = \lambda_{0}(G)$. \qed 
\end{proof}\medskip

By virtue of Buser's lemma, we derive the following consequence of the preceding characterization.

\begin{corollary}\label{unim and amen domain}
Let $G$ be a non-compact, connected, unimodular and amenable Lie group endowed with a left-invariant metric. Then for any $\varepsilon, r > 0$, there exists an open, bounded $W \subset G$ such that
\[
| (\partial W)^{r}| < \varepsilon | W \smallsetminus (\partial W)^{r} |.
\]
\end{corollary}

\section{Submersions with fibers of basic mean curvature}\label{sec basic mean}

The aim of this section is to prove Theorem \ref{basic mean curv thm}. Let $p \colon M_{2} \to M_{1}$ be a Riemannian submersion with fibers of basic mean curvature, and consider the Schr\"{o}dinger operator
\[
S = \Delta + \frac{1}{4} \| p_{*}H \|^{2} - \frac{1}{2} \diver p_{*}H
\]
on $M_{1}$. As in \cite{Bessa, Bordoni, Mine4}, the \textit{average} of a function $f \in C^{\infty}_{c}(M_{2})$ is the smooth function
\[
f_{\av}(x) := \int_{F_{x}} f
\]
on $M_{1}$ with gradient given by
\begin{equation}\label{grad of average}
\langle \grad f_{\av}(x)  , X \rangle= \int_{F_{x}} \langle \grad f - f H , \tilde{X} \rangle
\end{equation}
for any $x \in M_{1}$ and $X \in T_{x}M_{1}$, where $\tilde{X}$ is the horizontal lift of $X$ on $F_{x}$. The \textit{pushdown} of $f$ is the function
\[
h(x) := \sqrt{(f^{2})_{\av}(x)} = \left( \int_{F_{x}} f^{2} \right)^{1/2}
\]
on $M_{1}$. Then \cite[Lemma 3.1]{Mine4} states that $h \in \Lip_{c}(M_{1})$. Hence, its gradient is defined almost everywhere and vanishes (if defined) in points where $h$ is zero.

\begin{proposition}\label{rayleigh pushdown}
Let $h \in \Lip_{c}(M_{1})$ be the pushdown of a function $f \in C^{\infty}_{c}(M_{2})$ with $\| f \|_{L^{2}(M_{2})} = 1$. Then their Rayleigh quotients are related by
\[
\mathcal{R}(f) \geq \mathcal{R}_{S}(h) + \int_{M_{1}} \lambda_{0}(F_{x}) h^{2}(x) dx.
\]
\end{proposition}

\begin{proof}
For any $x \in M_{1}$ with $h(x) > 0$, we readily see from formula (\ref{grad of average}) that
\[
\grad h(x) =\frac{1}{2h(x)} \int_{F_{x}} (p_{*} \grad f^{2} - f^{2} p_{*}H) = \frac{1}{h(x)} \int_{F_{x}} f p_{*} \grad f - \frac{1}{2} h(x) p_{*}H(x).
\]
In view of this, the fact that $\| h \|_{L^{2}(M_{1})} = 1$, the divergence formula, the Cauchy-Schwarz inequality and that
\[
\frac{1}{2} \langle \grad h^{2}(x) , p_{*}H(x) \rangle = h(x) \langle \grad h(x) ,p_{*}H(x) \rangle = \int_{F_{x}} f \langle \grad f, H \rangle - \frac{1}{2} h^{2}(x) \| p_{*}H(x) \|^{2}
\]
for any $x \in M_{1}$, we compute
\begin{eqnarray}\label{pushdown}
\mathcal{R}_{S}(h) &=& \int_{M_{1}} \big( \| \grad h \|^{2} +\frac{1}{4} \| p_{*}H \|^{2} h^{2} - \frac{1}{2} h^{2} \diver p_{*}H \big) \nonumber \\
&=& \int_{M_{1}} \bigg( \frac{1}{h^{2}} \Big\| \int_{F_{x}} f p_{*} \grad f  \Big\|^{2}  + \frac{1}{4} h^{2} \| p_{*}H \|^{2} - \int_{F_{x}} f \langle \grad f , H \rangle  \bigg) \nonumber \\
&+& \int_{M_{1}} \bigg( \frac{1}{4} h^{2} \| p_{*}H \|^{2}  + \frac{1}{2} \langle \grad h^{2} , p_{*}H \rangle \bigg) \nonumber \\
&\leq& \int_{M_{1}} \int_{F_{x}} \| (\grad f)^{h} \|^{2} = \int_{M_{2}} \| (\grad f)^{h} \|^{2}.
\end{eqnarray}

Since at any point of $M_{2}$ the tangent space of $M_{2}$ splits into the orthogonal sum of the horizontal and the vertical space, it is easily checked that (cf. \cite[Formula (6)]{Mine4})
\[
\mathcal{R}(f) = \int_{M_{2}} \| (\grad f)^{h} \|^{2} + \int_{M_{2}} \| (\grad f)^{v} \|^{2} \geq \int_{M_{2}} \| (\grad f)^{h} \|^{2} + \int_{M_{1}} \lambda_{0}(F_{x}) h^{2}(x).
\]
The proof is now completed by formula (\ref{pushdown}) and Proposition \ref{bottom}. \qed
\end{proof}\medskip

\noindent{\emph{Proof of Theorem \ref{basic mean curv thm}}:} From Propositions \ref{bottom} and \ref{rayleigh pushdown}, it is immediate verify the asserted inequality. Suppose now that the equality holds. Then there exists $(f_{n})_{n \in \mathbb{N}} \subset C^{\infty}_{c}(M_{2})$ with $\| f_{n} \|_{L^{2}(M_{2})} = 1$ and $\mathcal{R}(f_{n}) \rightarrow \lambda_{0}(M_{2})$, as follows from Proposition \ref{bottom}. Denote by $h_{n} \in \Lip_{c}(M_{1})$ the pushdown of $f_{n}$, $n \in \mathbb{N}$. Arguing as in the proof of \cite[Theorem 1.1]{Mine4}, using Proposition \ref{rayleigh pushdown} instead of \cite[Proposition 3.2]{Mine4}, we obtain that
\begin{equation}\label{minimizing pushdowns}
\mathcal{R}_{S}(h_{n}) \rightarrow \lambda_{0}(S) \text{ and } \int_{M_{1}} (\lambda_{0}(F_{x}) - \inf_{y \in M_{1}} \lambda_{0}(F_{y})) h_{n}^{2}(x)dx \rightarrow 0.
\end{equation}
Since $\lambda_{0}(S) \notin \sigma_{\ess}(S)$, we deduce from \cite[Proposition 3.5]{Mine2} that after passing to a subsequence if necessary, we may assume that $h_{n} \rightarrow \varphi$ in $L^{2}(M_{1})$ for some $\varphi \in C^{\infty}(M_{1})$ with $S \varphi = \lambda_{0}(S) \varphi$. Then $\varphi$ is positive, by Proposition \cite[Proposition 3.7]{Mine2}. Arguing as in the proof of \cite[Theorem 1.1]{Mine4}, we conclude from (\ref{minimizing pushdowns}) that
\[
\lambda_{0}(F_{x}) = \inf_{y \in M_{1}} \lambda_{0}(F_{y})
\]
for almost any $x \in M_{1}$.\qed \medskip

\noindent\emph{Proof of Corollary \ref{closed basic}:} If the submersion has closed fibers of basic mean curvature, then $S$ is written as
\[
S = \Delta - \frac{\Delta \sqrt{V}}{\sqrt{V}},
\]
where $V(x)$ is the volume of $F_{x}$ (cf. \cite[Section 4]{Mine4}). Thus, we may consider the renormalization $L$ of $S$ with respect to $\sqrt{V}$. Then Lemmas \ref{renormalized comp} and \ref{lift} imply that for any non-zero $f \in C^{\infty}_{c}(M_{1})$, its lift $\tilde{f} \in C^{\infty}_{c}(M_{2})$ satisfies
\begin{equation}\label{Rayleigh lift}
\mathcal{R}(\tilde{f}) = \frac{\int_{M_{2}} \| \grad \tilde{f} \|^{2}}{\int_{M_{2}} \tilde{f}^{2}} = \frac{\int_{M_{1}} \| \grad f \|^{2} V}{\int_{M_{1}} f^{2}V} = \mathcal{R}_{L}(f) = \mathcal{R}_{S}(f\sqrt{V}).
\end{equation}
We derive from Proposition \ref{bottom} that $\lambda_{0}(M_{2}) \leq \lambda_{0}(S)$, while the inverse inequality is a consequence of Theorem \ref{basic mean curv thm}.

About the second statement, choose an exhausting sequence $(K_{n})_{n \in \mathbb{N}}$ of $M_{1}$ consisting of smoothly bounded, compact domains. Then $(p^{-1}(K_{n}))_{n \in \mathbb{N}}$ is an exhausting sequence of $M_{2}$ consisting of smoothly bounded, compact domains, because the submersion has closed fibers. Applying Theorem \ref{basic mean curv thm} to the restriction of $p \colon M_{2} \smallsetminus p^{-1}(K_{n}) \to M_{1} \smallsetminus K_{n}$ over any connected component of $M_{1} \smallsetminus K_{n}$, together with Proposition \ref{bottom essential spectrum}, gives the estimate
\[
\lambda_{0}^{\ess}(M_{2}) = \lim_{n} \lambda_{0}(M_{2} \smallsetminus p^{-1}(K_{n})) \geq \lim_{n} \lambda_{0}(S, M_{1} \smallsetminus K_{n}) = \lambda_{0}^{\ess}(S). 
\]
From Proposition \ref{bottom essential spectrum}, there exists a sequence $(f_{n})_{n \in \mathbb{N}} \subset C^{\infty}_{c}(M_{1}) \smallsetminus \{0\}$ with $\supp f_{n}$ pairwise disjoint, such that $\mathcal{R}_{S}(f_{n}) \rightarrow \lambda_{0}^{\ess}(S)$. It is clear that the lifts $\tilde{h}_{n}$ of $h_{n} := f_{n} / \sqrt{V}$ also have pairwise disjoint supports. Then Proposition \ref{bottom essential spectrum} and formula (\ref{Rayleigh lift}) yield that
\[
\lambda_{0}^{\ess}(M_{2}) \leq \liminf_{n} \mathcal{R}(\tilde{h}_{n}) = \liminf_{n} \mathcal{R}_{S}(f_{n}) = \lambda_{0}^{\ess}(S),
\]
as we wished. \qed \medskip

It should be noticed that if the submersion has minimal fibers, then $S$ coincides with the Laplacian on $M_{1}$. Therefore, \cite[Example 3.3]{Mine4} is an example of a Riemannian submersion $p \colon M_{2} \to M_{1}$ with minimal fibers, where $M_{1}$ is closed and $M_{2}$ is complete, such that
\[
0 = \lambda_{0}(M_{2}) = \lambda_{0}(S) + \inf_{x \in F_{x}} \lambda_{0}(F_{x}) 
\]
and there is $x \in M_{1}$ with $\lambda_{0}(F_{x}) > 0$. Since $M_{1}$ is closed, it is evident that $\lambda_{0}(S) \notin \sigma_{\ess}(S)$. Hence, in general, the asserted equality in the second part of Theorem \ref{basic mean curv thm} holds almost everywhere, but not everywhere.

According to the next lemma, the Schr\"{o}dinger operator $S$ defined in (\ref{operator}) is always non-negative definite. Moreover, it demonstrates that Theorem \ref{basic mean curv thm} provides a sharper lower bound for $\lambda_{0}(M_{2})$ than \cite[Theorem 1.1]{Mine4} in the case where both of them are applicable.

\begin{lemma}\label{non-negative def}
Let $X$ be a smooth vector field on a Riemannian manifold $M$. Then the operator
\[
S = \Delta + \frac{1}{4} \| X \|^{2} - \frac{1}{2} \diver X
\]
is non-negative definite. Furthermore, if $\| X \| \leq C \leq 2 \sqrt{\lambda_{0}(M)}$, then
\[
\lambda_{0}(S) \geq (\sqrt{\lambda_{0}(M)} - C/2)^{2}.
\]
\end{lemma}

\begin{proof}
For any $f \in C^{\infty}_{c}(M)$ with $\| f \|_{L^{2}(M)} = 1$, observe that its Rayleigh quotient is given by
\begin{equation}\label{special Rayleigh}
\mathcal{R}_{S}(f) = \int_{M} \big( \| \grad f \|^{2} + \frac{1}{4} \| X \|^{2} f^{2} + \langle \grad f , f X  \rangle \big) = \int_{M} \big\| \grad f + \frac{f}{2} X \big\|^{2},
\end{equation}
where we used the divergence formula. From Proposition \ref{bottom}, we readily see that $S$ is non-negative definite.
	
Suppose now that $\| X \| \leq C \leq 2 \sqrt{\lambda_{0}(M)}$ and let $f \in C^{\infty}_{c}(M)$ with $\| f \|_{L^{2}(M)} =1$. An elementary calculation shows that
\begin{eqnarray}\label{estimate}
\mathcal{R}_{S}(f) &\geq&  \int_{M} \big(\| \grad f \| - \frac{|f|}{2}\| X \|\big)^{2} \nonumber \\
&=& \int_{M} \big( \| \grad f \|^{2} + \frac{f^{2}}{4}\| X \|^{2} - \| \grad f \| |f| \|X\| \big) \nonumber \\
&\geq& \mathcal{R}(f) +\frac{1}{4} \int_{M} f^{2}\| X \|^{2} - \mathcal{R}(f)^{1/2} \left( \int_{M} f^{2} \| X \|^{2}  \right)^{1/2} \nonumber\\
&=& \left(\sqrt{\mathcal{R}(f)} - \frac{1}{2} \bigg( \int_{M} f^{2} \| X \|^{2}  \bigg)^{1/2}\right)^{2}.
\end{eqnarray}

By the assumption that $\| X \| \leq C \leq 2 \sqrt{\lambda_{0}(M)}$, the fact that $\| f \|_{L^{2}(M)} = 1$ and Proposition \ref{bottom}, we obtain that
\[
\sqrt{\mathcal{R}(f)} - \frac{1}{2} \bigg( \int_{M} f^{2} \| X \|^{2}  \bigg)^{1/2} \geq \sqrt{\lambda_{0}(M_{1})} - C/2 > 0.
\]
The proof is completed by Proposition \ref{bottom} and formula (\ref{estimate}).\qed
\end{proof}\medskip

We end this section with some simple examples where Theorem \ref{basic mean curv thm} and Corollary \ref{closed basic} can be applied.

\begin{examples}\label{examples}
\begin{enumerate}[topsep=0pt,itemsep=-1pt,partopsep=1ex,parsep=0.5ex,leftmargin=*, label=(\roman*), align=left, labelsep=0em]
\item The \textit{warped product} $M_{2} = M_{1} \times_{\psi} F$ is the product manifold endowed with the Riemannian metric $g_{M_{1}} \times \psi^{2} g_{F}$, where $\psi \in C^{\infty}(M_{1})$ is positive. Then the projection to the first factor $p \colon M_{2} \to M_{1}$ is a Riemannian submersion with fibers of basic mean curvature
\[
H= - k \grad \ln \tilde{\psi},
\]
where $k := \dim F$. In this case, the operator $S$ defined in (\ref{operator}) is written as
\[
S = \Delta - \frac{\Delta \psi^{k/2}}{\psi^{k/2}},
\]
and Theorem \ref{basic mean curv thm} says that
\[
\lambda_{0}(M_{2}) \geq \lambda_{0}(S) + \inf_{x\in M_{1}} \lambda_{0}(F_{x}) = \lambda_{0}(S) + \lambda_{0}(F) \inf_{x \in M_{1}} \psi^{-2}(x).
\]
If, in addition, $F$ is closed, then we deduce from Corollary \ref{closed basic} that $\lambda_{0}(M_{2}) = \lambda_{0}(S)$ and $\lambda_{0}^{\ess}(M_{2}) = \lambda_{0}^{\ess}(S)$. In particular, $M_{2}$ has discrete spectrum if and only if the spectrum of $S$ is discrete (compare with \cite[Theorem 3.3]{MR577877}). It noteworthy that surfaces of revolution are warped products of the form $\mathbb{R} \times_{\psi} S^{1}$.

\item Another generalization of surfaces of revolution was introduced by Bishop motivated by a result of Clairaut on such surfaces. A Riemannian submersion $p \colon M_{2} \to M_{1}$ is called \textit{Clairaut submersion} if there exists a positive $f \in C^{\infty}(M_{2})$, such that for any geodesic $c$
on $M_{2}$, the function $(f \circ c) \sin \theta$ is constant, where $\theta (t)$ is the angle between $c^{\prime}(t)$ and ($T_{c(t)}M_{2})^{h}$. Bishop proved that a Riemannian submersion $p \colon M_{2} \to M_{1}$ of complete manifolds with connected fibers, is a Clairaut submersion if and only if the fibers are totally umbilical with mean curvature
\[
H = - k \grad \ln \tilde{\psi}
\]
for some positive $\psi \in C^{\infty}(M_{1})$, where $k$ is the dimension of the fiber (cf. for instance \cite[Theorem 1.7]{MR2110043}). It is not difficult to establish statements for Clairaut submersions of complete manifolds with connected fibers analogous to the ones of the preceding example.

\item Let $p \colon M_{2} \to M_{1}$ be a Riemannian submersion \textit{arising from the action of a Lie group} $G$. That is, $G$ is a Lie group acting smoothly, freely and properly via isometries on a Riemannian manifold $M_{2}$, where $\dim G < \dim M_{2}$. Then $M_{1} := M_{2}/G$ is a Riemannian manifold and the projection $p \colon M_{2} \to M_{1}$ is a Riemannian submersion with fibers of basic mean curvature. In view of Theorem \ref{basic mean curv thm}, the bottoms of the spectra satisfy $\lambda_{0}(M_{2}) \geq \lambda_{0}(S)$. As a consequence of Corollary \ref{closed basic}, if $G$ is compact, then $\lambda_{0}(M_{2}) = \lambda_{0}(S)$.
\end{enumerate}
\end{examples}

\section{Submersions arising from Lie group actions}\label{sec principal}

Throughout this section, we consider a Riemannian submersion $p \colon M_{2} \to M_{1}$ arising from the action of a Lie group $G$. For convenience of the reader, we provide a brief outline of the section and the proof of Theorem \ref{submersion group thm}.

In Subsection \ref{induced metrics subsec}, we show that identifying the fiber with $G$ along a section of the submersion gives rise to a smooth family of left-invariant metrics on $G$. This remark plays a quite important role in our discussion. More precisely, from this and Theorem \ref{basic mean curv thm}, we obtain Theorem \ref{submersion group thm}(iii).

The other assertions of Theorem \ref{submersion group thm} are first proved in the case where $G$ is connected. If $G$ is compact, then they follow from Corollary \ref{closed basic} and \cite[Theorem 1.2]{Mine4}. Thus, we have to focus on the case where $G$ is non-compact and connected. In Subsection \ref{Partition of Unity}, we construct cut-off functions on such $G$ closely related to the open sets $W$ from Corollary \ref{unim and amen domain}. In terms of these functions, for a section $s \colon U \subset M_{1} \to M_{2}$, we define cut-off functions in $p^{-1}(U)$ with uniformly (that is, independently from the corresponding $W$) bounded gradient and Laplacian.

We begin Subsection \ref{pulling up subsec} with the proposition that establishes this auxiliary result. The main idea is that given an $f \in C^{\infty}_{c}(M_{1})$, we may write it as a sum of functions supported in domains admitting sections. Using cut-off functions as above, we are able to pull up these functions, and for suitable choices of $W$, the sum of these pulled up functions coincides with the lift of $f$ in a relatively large part of its support. In the rest of its support, its gradient and its Laplacian are bounded independently from $W$.

The proof of Theorem \ref{submersion group thm} is completed after observing that such a submersion $p$ is expressed as the composition of the submersion arising from the action of $G_{0}$ with the covering arising from the action of $G/G_{0}$.

\subsection{Induced metrics on the Lie group}\label{induced metrics subsec}

Let $p \colon M_{2} \to M_{1}$ be a Riemannian submersion arising from the action of a (possibly non-connected) Lie group $G$. Given a section $s \colon U \subset M_{1} \to M_{2}$, it is easily checked that the map $\Phi \colon G \times U \to p^{-1}(U)$ defined by $\Phi (x,y) := x s(y)$ is a diffeomorphism, and so is its restriction $\Phi_{y} := \Phi(\cdot, y) \colon G \to F_{y}$. Denote by $g_{s(y)}$ the metric induced on $G$ via $\Phi_{y}$, that is, the pullback via $\Phi_{y}$ of the restriction of the metric of $M_{2}$ on $F_{y}$. It is straightforward to see that the metric $g_{s(y)}$ depends only on $s(y)$ and not on the behavior of $s$ in a neighborhood of $y$.

\begin{proposition}\label{induced metric}
Let $s \colon U \subset M_{1} \to M_{2}$ be a section. Then	the Riemannian metric $g_{s(y)}$ is left-invariant and depends smoothly on $y \in U$.
\end{proposition}

\begin{proof}
For $x_{1} , x_{2} \in G$, it is immediate to verify that
\[
x_{1} \Phi_{y}(x_{2}) = x_{1} x_{2} s(y) = \Phi_{y}(x_{1}x_{2}),
\]
and therefore, ${x_{1}}_{*} {\Phi_{y}}_{*} = {\Phi_{y}}_{*} {L_{x_{1}}}_{*}$. Bearing in mind that $G$ acts on $M_{2}$ via isometries, given $x \in G$ and $X,Y \in T_{e}G$, where $e$ is the neutral element of $G$, it is now elementary to compute
\begin{eqnarray}
g_{s(y)}({L_{x}}_{*}X , {L_{x}}_{*} Y)(x) &=& \langle {\Phi_{y}}_{*} {L_{x}}_{*}X ,{\Phi_{y}}_{*} {L_{x}}_{*} Y \rangle_{\Phi_{y}(x)} = \langle { x_{*} \Phi_{y}}_{*} X ,x_{*} {\Phi_{y}}_{*} Y \rangle_{x \Phi_{y}(e)} \nonumber \\
&=& \langle { \Phi_{y}}_{*} X , {\Phi_{y}}_{*} Y \rangle_{\Phi_{y}(e)}
= g_{s(y)}( X, Y ) (e), \nonumber
\end{eqnarray}
which yields that the induced metric on $G$ is left-invariant.

Choose a left-invariant frame field $\{X_{i}\}_{i=1}^{k}$ on $G$. After endowing $G$ with a left-invariant metric and $G \times U$ with the product metric, it is evident that the projection to the first factor $q \colon G \times U \to G$ is a Riemannian submersion. Consider the horizontal lift $\tilde{X}_{i}$ of $X_{i}$ on $G \times U$. Notice that $\{\tilde{X}_{i}\}_{i=1}^{k}$ is a $G$-invariant, smooth frame field, and hence, so is $\{\Phi_{*}\tilde{X}_{i}\}_{i=1}^{k}$. Then for $y \in U$ and $x \in G$, we deduce that
\[
g_{s(y)}(X_{i} , X_{j})(x) = g_{s(y)}(X_{i},X_{j})(e) = \langle \Phi_{*}\tilde{X}_{i} , \Phi_{*}\tilde{X}_{j} \rangle_{s(y)}.
\]
Since $\langle \Phi_{*}\tilde{X}_{i} , \Phi_{*}\tilde{X}_{j} \rangle_{z}$ is a smooth function (with respect to $z$) in $p^{-1}(U)$, so is its composition with $s$, as we wished. \qed
\end{proof}

\begin{corollary}\label{volume element}
Let $s \colon U \subset M_{1} \to M_{2}$ be a section and fix a left-invariant metric $g$ on $G$. Then there exists $V_{s} \in C^{\infty}(U)$ such that for any $y \in U$, the volume element of the induced metric satisfies
\[
d {\vol}_{g_{s(y)}} = V_{s}(y) d {\vol}_{g}.
\]
\end{corollary}

\begin{proof}
Follows from Proposition \ref{induced metric} and the local expression of the volume element. \qed
\end{proof}\medskip

For $y \in M_{1}$ and $z_{1} , z_{2} \in F_{y}$, consider the diffeomorhisms $\Phi_{i} \colon G \to F_{y}$ defined by $\Phi_{i}(x) = x z_{i}$, and the induced metrics $g_{i} := g_{z_{i}}$ on $G$, $i=1,2$. Because $G$ acts transitively on $F_{y}$, there exists $x_{0} \in G$ such that $x_{0} z_{1} = z_{2}$. Then it is apparent that
\[
\Phi_{2}(x) = x z_{2} = x x_{0} z_{1} = (\Phi_{1} \circ R_{x_{0}}) (x).
\]
In particular, if $G$ is unimodular, then we have that
\begin{equation}\label{volume eq}
d{\vol}_{g_{2}} = \Phi_{2}^{*} (d{\vol}_{F_y}) = R_{x_{0}}^{*}(\Phi_{1}^{*} (d{\vol}_{F_{y}})) = R_{x_{0}}^{*}(d{\vol}_{g_{1}}) = d{\vol}_{g_{1}},
\end{equation}
where $d\vol_{F_{y}}$ is the volume element of $F_{y}$ with respect to the induced metric from $M_{2}$. This implies that the function $V_{s}$ from Corollary \ref{volume element} is independent from the section $s$ and can be defined globally.

\begin{corollary}\label{glob defined V}
Suppose that $G$ is unimodular and fix a left-invariant metric $g$ on $G$. Then there exists $V \in C^{\infty}(M_{1})$ such that for any section $s \colon U \subset M_{1} \to M_{2}$ and $y \in U$, the volume element of the induced metric satisfies
\[
d{\vol}_{g_{s(y)}} = V(y) d{\vol}_{g}.
\]
Moreover, the gradient of $V$ is given by
\[
\grad V = - V p_{*}H.
\]
\end{corollary}

\begin{proof}
The existence of the function $V$ is a consequence of Corollary \ref{volume element} and formula (\ref{volume eq}). About the second statement, let $y \in M_{1}$ and $s \colon U \subset M_{1} \to M_{2}$ be a section defined in a neighborhood $U$ of $y$ that is horizontal at $y$, which means that $s_{*}T_{y}M_{1}$ is the horizontal space of $M_{2}$ at $s(y)$. Let $X \in T_{y}M_{1}$ and $c \colon (- \varepsilon , \varepsilon) \to M_{1}$ be a smooth curve with $c(0) = y$ and $c^{\prime}(0) = X$. Denote by $F \colon (- \varepsilon , \varepsilon) \times G \to M_{2}$ the smooth variation of the isometric immersion $F(0, \cdot) \colon (G,g_{s(y)}) \to M_{2}$ defined by $F(t,x) = x s(c(t))$,
and observe that its variational vector field is the horizontal lift $\tilde{X}$ of $X$ on $F_{y}$. The asserted equality follows now from the first variational formula.\qed
\end{proof}\medskip

It is well known that if $p \colon M_{2} \to M_{1}$ is a Riemannian submersion and $M_{2}$ is complete, then so is $M_{1}$. According to the next corollary, if the submersion arises from the action of a Lie group, the converse implication is also true.

\begin{corollary}\label{complete}
Let $p \colon M_{2} \to M_{1}$ be a Riemannian submersion arising from the action of a Lie group $G$. If $M_{1}$ is complete, then $M_{2}$ is complete.
\end{corollary}

\begin{proof}
Fix a left-invariant metric $g$ on $G$ and let $(z_{n})_{n \in \mathbb{N}} \subset M_{2}$ be a Cauchy sequence. Then $(p(z_{n}))_{n \in \mathbb{N}}$ is a Cauchy sequence in $M_{1}$, and hence, $p(z_{n}) \rightarrow y$ for some $y \in M_{1}$. Let $s \colon U \subset M_{1} \to M_{2}$ be a section defined in a neighborhood $U$ of $y$, and consider the corresponding diffeomorphism $\Phi \colon G \times U \to M_{2}$, as in the beginning of this subsection.
Without loss of generality, we may assume that $z_{n} \in p^{-1}(U)$ for any $n \in \mathbb{N}$. 
Writing $z_{n} = \Phi(x_{n},p(z_{n}))$, it remains to prove that $(x_{n})_{n \in \mathbb{N}}$ converges. Given a precompact neighborhood $U_{y}$ of $y$ with $\bar{U}_{y} \subset U$, it simple to see that for any sufficiently small $\varepsilon > 0$, there is $n_{0} \in \mathbb{N}$, such that for any $n,m \geq n_{0}$ there exists a smooth curve $c_{n,m}$ from $z_{n}$ to $z_{m}$ of length less than $\varepsilon$, whose image is contained in $p^{-1}(U_{y})$. Denoting by $q \colon G \times U \to G$ the projection to the first factor, it is clear that $\hat{c}_{n,m} := q \circ \Phi^{-1} \circ c_{n,m}$ is a smooth curve from $x_{n}$ to $x_{m}$. Since $U_{y}$ is precompact, we derive from Proposition \ref{induced metric} that there exists $C>0$ such that $\ell_{g}(\hat{c}_{n,m}) \leq C \ell(c_{n,m})$ for any $n,m \geq n_{0}$, where $\ell(\cdot)$ stands for the length of a curve. This shows that $(x_{n})_{n \in \mathbb{N}}$ is Cauchy in $(G,g)$, and thus, converges. \qed
\end{proof}

\subsection{Cut-off Functions}\label{Partition of Unity}

The aim of this subsection is to construct some special functions on the Lie group that will be used in the sequel to obtain cut-off functions on $M_{2}$. Throughout this subsection, we consider a non-compact, connected Lie group $G$ endowed with a left-invariant metric. Given $r > 0$, choose a sequence $(x_{n})_{n \in \mathbb{N}} \subset G$ such that $d(x_{n},x_{m}) \geq r$ for any $n \neq m$ and the open balls $B(x_{n},r)$ cover $G$.

\begin{lemma}\label{Covering Lemma}
There exists $n(r) \in \mathbb{N}$ such that any $x \in G$ lies in at most $n(r)$ of the balls $B(x_{n},2r)$, with $n \in \mathbb{N}$.
\end{lemma}

\begin{proof}
Let $x \in G$ and set $E_{x} := \{ n \in \mathbb{N} : x \in B(x_{n},2r) \}$. Notice that for $n \in E_{x}$, we have that $B(x_{n},r/2) \subset B(x,5r/2)$ and the balls $B(x_{n},r/2)$ are disjoint. Bearing in mind that $G$ is a homogeneous space, we compute
\[
|B(x,5r/2)| \geq \sum_{n \in E_{x}} |B(x_{n},r/2)| = |E_{x}| |B(x,r/2)|,
\]
where $|E_{x}|$ is the cardinality of $E_{x}$. Since the Ricci curvature of $G$ is bounded from below (say by $(k-1)C$, where $k$ is the dimension of $G$), the Bishop-Gromov volume comparison theorem gives the estimate
\[
|E_{x}| \leq \frac{|B(x,5r/2)|}{|B(x,r/2)|} \leq \frac{|B_{5r/2}|}{|B_{r/2}|} =: n(r),
\]
where $B_{\rho}$ is a ball of radius $\rho$ in the $k$-dimensional space form of sectional curvature $C$. \qed
\end{proof}\medskip

Fix $\psi_{e} \in C^{\infty}_{c}(G)$ with $0\leq \psi_{e} \leq 1$, $\supp \psi_{e} \subset B(e,3r/2)$ and $\psi_{e} = 1$ in $B(e,r)$. For $n \in \mathbb{N}$, the function $\psi_{n} := \psi_{e} \circ L_{x_{n}}^{-1}$ satisfies $0\leq \psi_{n} \leq 1$, $\supp \psi_{n} \subset B(x_{n},3r/2)$ and $\psi_{n} = 1$ in $B(x_{n},r)$.
We know from Lemma \ref{Covering Lemma} that the cover $\{B(x_{n},3r/2)\}_{n \in \mathbb{N}}$ is locally finite, which implies that the function $\psi := \sum_{n \in \mathbb{N}} \psi_{n}$ is well-defined and smooth. It is evident that $\psi \geq 1$, $G$ being covered by the balls $B(x_{n} , r)$. The smooth partition of unity on $G$ consisting of the functions $\zeta_{n} := \psi_{n}/\psi$ with $n \in \mathbb{N}$, is called a \textit{partition of unity corresponding to} $r$.
Apparently, any point of $G$ lies in at most $n(r)$ of the supports of $\zeta_{n}$, where $n(r)$ is the constant from Lemma \ref{Covering Lemma}. 
The cut-off function corresponding to a subset $E$ of $\mathbb{N}$ is defined by
\[
\chi_{E} := \sum_{n \in E} \zeta_{n}.
\]

Let $p \colon M_{2} \to M_{1}$ be a Riemannian submersion arising from the action of a non-compact, connected Lie group $G$. Consider a relatively compact, open domain $U \subset M_{1}$ that admits an extensible section $s \colon U \to M_{2}$, and the corresponding diffeomorphism $\Phi \colon G \times U \to p^{-1}(U)$ defined by $\Phi(x,y) := x s(y)$. For a function $f \colon G \to \mathbb{R}$, we denote by $f_{s} \colon p^{-1}(U) \to \mathbb{R}$ the function satisfying
\[
f_{s}(\Phi(x,y)) := f(x)
\]
for any $x \in G$ and $y \in U$. Given a left-invariant metric on $G$ and $r > 0$, we consider a partition of unity on $G$ corresponding to $r$ and the functions $\chi_{E}$ for $E \subset \mathbb{N}$.

\begin{lemma}\label{uniform estimate}
Let $s \colon U \to M_{2}$ be an extensible section defined on a precompact domain $U$ of $M_{1}$. Then there exists a constant $C$ independent from $E \subset \mathbb{N}$, such that
\[
| \Delta (\chi_{E})_{s} (z) | \leq C \text{ and } \| \grad (\chi_{E})_{s} (z) \| \leq C
\]
for any $z \in p^{-1}(U)$.
\end{lemma}

\begin{proof}
Since $U$ is precompact and $s$ is extensible, it is easily checked that the Laplacian and the gradient of $(\psi_{e})_{s}$ are bounded. Since $(\psi_{n})_{s}(z) = (\psi_{e})_{s}(x_{n}^{-1}z)$ for any $n \in \mathbb{N}$ and $z \in p^{-1}(U)$, we obtain uniform estimates for the Laplacian and the gradient of $(\psi_{n})_{s}$ for all $n \in \mathbb{N}$. Then Lemma \ref{Covering Lemma} yields a uniform bound for the Laplacian and the gradient of the functions $\sum_{n \in E} (\psi_{n})_{s}$ for all subsets $E \subset \mathbb{N}$. The proof is completed after observing that
\[
(\chi_{E})_{s} = \frac{\sum_{n \in E} (\psi_{n})_{s}}{\sum_{n \in \mathbb{N}} (\psi_{n})_{s}}
\]
and that $\sum_{n \in \mathbb{N}} (\psi_{n})_{s} \geq 1$.\qed
\end{proof}\medskip

The purpose of considering this partition of unity becomes more clear in the next proposition, where we combine this construction with Corollary \ref{unim and amen domain} in the case where $G$ is unimodular and amenable.

\begin{proposition}\label{Cut-Off}
Let $G$ be a non-compact, connected, unimodular and amenable Lie group endowed with a left-invariant metric. Consider $r > 0$ and a partition of unity $\{\zeta_{n}\}_{n\in \mathbb{N}}$ corresponding to $r/2$. Then for any $\varepsilon>0$, there exists an open, bounded $W \subset G$ and a finite $E \subset \mathbb{N}$, such that $\chi_{E} = 1$ in $W \smallsetminus (\partial W)^{r}$, $\supp \chi_{E} \subset W^{r/2}$ and
\[
|(\partial W)^{2r}| < \varepsilon |W \smallsetminus (\partial W)^{2r}|.
\]
\end{proposition}

\begin{proof}
As a consequence of Corollary \ref{unim and amen domain}, for any $\varepsilon > 0$, there exists an open, bounded $W \subset G$ such that the desired inequality for the volumes holds. Consider the finite set $E := \{ n \in \mathbb{N} : x_{n} \in W \smallsetminus (\partial W)^{r/4} \}$. It is elementary to verify that if $x \in W \smallsetminus (\partial W)^{r}$, then $n \in E$ for any $n \in \mathbb{N}$ with $x \in B(x_{n},3r/4)$, and therefore, $\chi_{E} = 1$ in $W \smallsetminus (\partial W)^{r}$. From the fact that $\supp \zeta_{n} \subset B(x_{n},3r/4)$, it follows that $\supp \chi_{E} \subset W^{r/2}$.\qed
\end{proof}

\subsection{Pulling up}\label{pulling up subsec}

Suppose now that $G$ is unimodular and let $V$ be the function from Corollary \ref{glob defined V}. A straightforward calculation shows that
\[
S = \Delta + \frac{1}{4} \| p_{*}H \|^{2} - \frac{1}{2} \diver p_{*}H = \Delta - \frac{\Delta \sqrt{V}}{\sqrt{V}}.
\]
This allows us to consider the renormalization
\[
L = m_{\sqrt{V}}^{-1} \circ S \circ m_{\sqrt{V}} = \Delta - \grad \ln V = \Delta + p_{*}H
\]
of $S$ with respect to $\sqrt{V}$, where we used again Corollary \ref{glob defined V}. According to Lemma \ref{lift}, the Laplacian of the lift $\tilde{f}$ of any $f \in C^{\infty}(M_{1})$ is given by
\begin{equation}\label{lift eq}
\Delta \tilde{f} = \widetilde{Lf}.
\end{equation}

\begin{proposition}\label{Pulling Up}
Let $p \colon M_{2} \to M_{1}$ be a Riemannian submersion arising from the action of a non-compact, connected, unimodular and amenable Lie group $G$. Then for any $\lambda \in \mathbb{R}$, $\varepsilon > 0$ and $f \in C^{\infty}_{c}(M_{1}) \smallsetminus \{0\}$, there exists $h \in C^{\infty}_{c}(M_{2}) \smallsetminus \{0\}$, such that
\[
\frac{\| (\Delta - \lambda)h \|_{L^{2}(M_{2})}^{2}}{\| h \|_{L^{2}(M_{2})}^{2}} \leq \frac{\| (L - \lambda)f \|_{L^{2}_{\sqrt{V}}(M_{1})}^{2}}{\| f \|_{L^{2}_{\sqrt{V}}(M_{1})}^{2}} + \varepsilon.
\]
\end{proposition}

\begin{proof}
Cover $\supp f$ with finitely many open, precompact domains $U_{i}$ that admit extensible sections $s_{i} \colon U_{i} \to M_{2}$, $i=1, \dots, k$, and choose non-negative $\varphi_{i} \in C^{\infty}_{c}(U_{i})$ with $\sum_{i=1}^{k} \varphi_{i} = 1$ in $\supp f$. Denote by $x_{ij} \colon U_{i} \cap U_{j} \to G$ the transition maps defined by $s_{j}(y) = x_{ij}(y)s_{i}(y)$ for all $y \in U_{i} \cap U_{j}$, and by $\Phi_{i} \colon G \times U_{i} \to p^{-1}(U_{i})$ the diffeomorphisms defined by $\Phi_{i}(x,y) = x s_{i}(y)$.

Fix a left-invariant metric $g$ on $G$. Since $U_{i}$ is precompact and $s_{i}$ is extensible, notice that there exists $r > 0$ such that $x_{ij}(U_{i} \cap U_{j}) \subset B_{g}(e,r)$ for any $i,j=1,\dots,k$. Let $\{\zeta_{n}\}_{n \in \mathbb{N}}$ be a partition of unity on $G$ corresponding to $r/2$, as in Subsection \ref{Partition of Unity}. For a finite subset $E$ of $\mathbb{N}$, consider the compactly supported, smooth function
\[
h_{i} := (\chi_{E})_{s_{i}} \tilde{\varphi}_{i} \tilde{f}
\]
in $p^{-1}(U_{i})$, $i=1 , \dots , k$. Setting $h = \sum_{i=1}^{k} h_{i}$, we derive from Lemma \ref{uniform estimate}, that there exists a constant $C$ independent from $E$, such that $| (\Delta - \lambda) h (z) | \leq C$ for any $z \in M_{2}$.

We know from Proposition \ref{Cut-Off} that for any $\varepsilon > 0$, there exists an open, bounded $W \subset G$ and a finite $E \subset \mathbb{N}$, such that $\chi_{E} = 1$ in $W \smallsetminus (\partial W)^{r}$, $\supp \chi_{E} \subset W^{r/2}$ and
\begin{equation}\label{volume estimate}
\frac{|W_{0}^{\prime}|_{g}}{|W_{0}|_{g}} < \frac{\varepsilon \| f \|_{L^{2}_{\sqrt{V}}}^{2}}{C^{2} \int_{\supp f} V},
\end{equation}
where $W_{0} := W \smallsetminus (\partial W)^{2r}$, $W_{0}^{\prime} := (\partial W)^{2r}$ and tubular neighborhoods are considered with respect to the background metric $g$. Denote by $W_{i}(y)$ and $W_{i}^{\prime}(y)$ the images of $W_{0}$ and $W_{0}^{\prime}$ via $\Phi_{i}(\cdot, y)$, respectively.
Bearing in mind that
\[
\Phi_{i}(x,y) = \Phi_{j}(xx_{ji}(y) , y)
\]
for any $y \in U_{i} \cap U_{j}$ and $x \in G$, it is not difficult to see that $h(z) = \tilde{f}(z)$ for any $z \in W_{i}(y)$ and that $\supp h \cap F_{y} \subset W_{i}(y) \cup W_{i}^{\prime}(y)$ for any $y \in U_{i}$, $i=1,\dots,k$. In view of Corollary \ref{glob defined V}, it is now simple to compute
\begin{eqnarray}
\| h \|_{L^{2}}^{2} &=& \sum_{i=1}^{k} \int_{M_{2}} \tilde{\varphi}_{i}h^{2} \geq \sum_{i=1}^{k} \int_{U_{i}} \int_{W_{i}(y)} \tilde{\varphi}_{i}h^{2} dy = \sum_{i = 1}^{k} \int_{U_{i}} \varphi_{i}(y) f^{2}(y)|W_{0}|_{g_{s_{i}(y)}} dy \nonumber \\
&=& |W_{0}|_{g} \sum_{i = 1}^{k} \int_{U_{i}} \varphi_{i} f^{2} V = |W_{0}|_{g} \| f \|_{L^{2}_{\sqrt{V}}}^{2} . \nonumber
\end{eqnarray}
Furthermore, it is apparent that
\begin{eqnarray}
\| (\Delta - \lambda) h \|_{L^{2}(M_{2})}^{2} &=& \sum_{i = 1}^{k} \int_{M_{2}} \tilde{\varphi}_{i} ((\Delta - \lambda) h)^{2} \nonumber \\
&=& \sum_{i= 1}^{k} \left( \int_{U_{i}} \int_{W_{i}(y)} \tilde{\varphi}_{i} ((\Delta - \lambda) h)^{2} dy + \int_{U_{i}} \int_{W_{i}^{\prime}(y)} \tilde{\varphi}_{i} ((\Delta - \lambda) h)^{2} dy\right). \nonumber
\end{eqnarray}
By virtue of Corollary \ref{glob defined V} and formula (\ref{lift eq}), we deduce that
\begin{eqnarray}
\sum_{i= 1}^{k} \int_{U_{i}} \int_{W_{i}(y)} \tilde{\varphi}_{i} ((\Delta - \lambda) h)^{2} dy &=& \sum_{i= 1}^{k} \int_{U_{i}} \int_{W_{i}(y)} \tilde{\varphi}_{i} ((\Delta - \lambda) \tilde{f})^{2} dy \nonumber \\
&=& \sum_{i= 1}^{k} \int_{U_{i}} \varphi_{i}(y) ((L - \lambda) f(y))^{2} |W_{0}|_{g_{s_{i}(y)}} dy \nonumber \\
&=& |W_{0}|_{g} \int_{M_{1}} ((L-\lambda)f)^{2} V \nonumber\\
&=& |W_{0}|_{g} \| (L - \lambda)f \|_{L^{2}_{\sqrt{V}}(M_{1})}^{2} . \nonumber
\end{eqnarray}
Finally, Corollary \ref{glob defined V} implies that
\[
\sum_{i=1}^{k}\int_{U_{i}}\int_{W_{i}^{\prime}(y)} \tilde{\varphi}_{i} ((\Delta - \lambda) h)^{2} \leq C^{2} \sum_{i=1}^{k} \int_{ \supp f \cap U_{i} } \varphi_{i}(y) |W^{\prime}_{0}|_{g_{s_{i}(y)}} dy 
= C^{2} |W^{\prime}_{0}|_{g} \int_{\supp f} V. 
\]
From the above estimates, we conclude that
\[
\frac{\| (\Delta - \lambda)h \|_{L^{2}(M_{2})}^{2}}{\| h \|_{L^{2}(M_{2})}^{2}} \leq \frac{\| (L - \lambda)f \|_{L^{2}_{\sqrt{V}}(M_{1})}^{2}}{\| f \|_{L^{2}_{\sqrt{V}}(M_{1})}^{2}} + \frac{|W_{0}^{\prime}|_{g}}{|W_{0}|_{g}} \frac{C^{2} \int_{\supp f} V}{\| f \|_{L^{2}_{\sqrt{V}}}^{2}},
\]
which, together with (\ref{volume estimate}), completes the proof. \qed
\end{proof}\medskip

Similarly, exploiting the second inequality of Lemma \ref{uniform estimate} instead of the first one, it is not hard to show the following:

\begin{proposition}\label{weaker pulling up}
Let $p \colon M_{2} \to M_{1}$ be a Riemannian submersion arising from the action of a non-compact, connected, unimodular and amenable Lie group $G$. Then for any $\varepsilon > 0$ and $f \in C^{\infty}_{c}(M_{1}) \smallsetminus \{0\}$, there exists $h \in C^{\infty}_{c}(M_{2}) \smallsetminus \{0\}$ such that $\mathcal{R}(h) \leq \mathcal{R}_{L}(f) + \varepsilon$.
\end{proposition}

Before proceeding to the proof of Theorem \ref{submersion group thm}, we establish a part of it in the case where $G$ is a connected Lie group.

\begin{proposition}\label{main connected}
Let $p \colon M_{2} \to M_{1}$ be a Riemannian submersion arising from the action of a connected Lie group $G$. If $G$ is unimodular and amenable, then $\lambda_{0}(M_{2}) = \lambda_{0}(S)$. If, in addition, $M_{1}$ is complete, then $\sigma(S) \subset \sigma(M_{2})$.
\end{proposition}

\begin{proof}
According to Corollary \ref{closed basic}, if $G$ is compact, then $\lambda_{0}(M_{2}) = \lambda_{0}(S)$. If, in addition, $M_{1}$ is complete, then Corollary \ref{complete} asserts that so is $M_{2}$, and the second statement is a consequence of \cite[Theorem 1.2]{Mine4}.

Suppose now that $G$ is non-compact, unimodular and amenable. Then for any $\varepsilon > 0$, there exists a non-zero $f \in C^{\infty}_{c}(M_{1})$ such that $\mathcal{R}_{S}(f) < \lambda_{0}(S) + \varepsilon/2$, by Proposition \ref{bottom}. From Propositions \ref{renormalized comp} and \ref{weaker pulling up}, it follows that there exists a non-zero $h \in C^{\infty}_{c}(M_{2})$ with
\[
\mathcal{R}(h) \leq \mathcal{R}_{L}(f/\sqrt{V}) + \varepsilon/2 = \mathcal{R}_{S}(f) + \varepsilon / 2 < \lambda_{0}(S) + \varepsilon.
\]
The proof of the first assertion is completed by Proposition \ref{bottom}, $\varepsilon > 0$ being arbitrary.

Assume now that, in addition, $M_{1}$ is complete and notice that $M_{2}$ is also complete, from Corollary \ref{complete}. Then Proposition \ref{characteristic seq} yields that for any $\lambda \in \sigma(S)$ there exists a characteristic sequence $(f_{n})_{n \in \mathbb{N}}$ for $S$ and $\lambda$. In view of Proposition \ref{Pulling Up} and Lemma \ref{renormalized comp}, for any $n \in \mathbb{N}$, there exists $h_{n}\in C^{\infty}_{c}(M_{2}) \smallsetminus \{0\}$ satisfying
\[
\frac{\| (\Delta - \lambda)h_{n} \|_{L^{2}(M_{2})}^{2}}{\| h_{n} \|_{L^{2}(M_{2})}^{2}} \leq \frac{\| (L - \lambda)(f/\sqrt{V}) \|_{L^{2}_{\sqrt{V}}(M_{1})}^{2}}{\| f / \sqrt{V} \|_{L^{2}_{\sqrt{V}}(M_{1})}^{2}} + \frac{1}{n} = \frac{\| (S - \lambda)f \|_{L^{2}(M_{1})}^{2}}{\| f \|_{L^{2}(M_{1})}^{2}} + \frac{1}{n} \rightarrow 0,
\]
as $n \rightarrow +\infty$. That is, $(h_{n})_{n \in \mathbb{N}}$ is a characteristic sequence for $\Delta$ (on $M_{2}$) and $\lambda$, and hence, $\lambda \in \sigma(M_{2})$, from Proposition \ref{characteristic seq}.\qed
\end{proof}\medskip

Consider now a Riemannian submersion $p \colon M_{2} \to M_{1}$ arising from the action of a Lie group $G$. Denote by $p_{1} \colon M_{2} \to M$ the Riemannian submersion arising from the action of the connected component $G_{0}$ of $G$. Then the action of $G$ on $M_{2}$ descends to a properly discontinuous action of $G/G_{0}$ on $M$, which gives rise to a Riemannian covering $p_{2} \colon M \to M_{1}$, and the original submersion is decomposed as $p = p_{2} \circ p_{1}$. It is immediate to verify that the Schr\"{o}dinger operator
\[
S_{M} :=\Delta + \frac{1}{4} \| p_{1*} H \|^{2} - \frac{1}{2} \diver p_{1*}H
\]
on $M$, defined as in (\ref{operator}), is the lift of the corresponding Schr\"{o}dinger operator $S$ on $M_{1}$.
\medskip

\noindent\emph{Proof of Theorem \ref{submersion group thm}:} Write $p = p_{2} \circ p_{1}$ as above, and suppose that $G$ is amenable and $G_{0}$ is unimodular. Then Lemma \ref{short exact} states that $G_{0}$ and $G/G_{0}$ are also amenable. From Proposition \ref{main connected} and \cite[Theorem 1.2]{BMP1}, we obtain that
\[
\lambda_{0}(M_{2}) = \lambda_{0}(S_{M}) = \lambda_{0}(S).
\]
If, in addition, $M_{1}$ is complete, then so is $M$, and the spectra are related by
\[
\sigma(S) \subset \sigma(S_{M}) \subset \sigma(M_{2}),
\]
where the first inclusion follows from \cite[Corollaries 4.21 and 4.22]{Mine} and the second one from Proposition \ref{main connected}.

Conversely, assume that $\lambda_{0}(M_{2}) = \lambda_{0}(S) \notin \sigma_{\ess}(S)$. By virtue of Theorem \ref{basic mean curv thm}, we have that $\lambda_{0}(F_{x}) = 0$ for almost any $x \in M_{1}$. Recall that $F_{x}$ is isometric to $G$ endowed with a left-invariant metric, from Lemma \ref{induced metric}. Taking into account that $\lambda_{0}(G) = \lambda_{0}(G_{0})$, we derive from Theorem \ref{unim and amen} that $G_{0}$ is unimodular and amenable. Moreover, Theorem \ref{basic mean curv thm} and \cite[Theorem 1.1]{BMP1} show that
\[
\lambda_{0}(M_{2}) \geq \lambda_{0}(S_{M}) \geq \lambda_{0}(S),
\]
and thus, $\lambda_{0}(S_{M}) = \lambda_{0}(S)$. Since $\lambda_{0}(S) \notin \sigma_{\ess}(S)$, we conclude from \cite[Theorem 1.2]{Mine2} that $p_{2}$ is an amenable covering, or equivalently, $G/G_{0}$ is amenable. The proof is completed by Lemma \ref{short exact}.
\qed\medskip

\noindent\emph{Proof of Corollary \ref{closed cor}:} Suppose first that $G$ is unimodular and amenable, and fix a left-invariant metric on $G$. By formula (\ref{special Rayleigh}), it is easily checked that $\mathcal{R}_{S}(\sqrt{V}) = 0$ for the positive $V \in C^{\infty}(M_{1})$ from Corollary \ref{glob defined V}, which together with Theorem \ref{submersion group thm}, Proposition \ref{bottom} and Lemma \ref{non-negative def}, implies that $\lambda_{0}(M_{2}) = \lambda_{0}(S) = 0$.

Conversely, assume that $\lambda_{0}(M_{2}) = 0$ and write $p = p_{1} \circ p_{2}$ as above. We readily see from Theorem \ref{basic mean curv thm} that $\lambda_{0}(S) = 0$. Then $G$ is amenable and $G_{0}$ is unimodular, from Theorem \ref{submersion group thm}, because $\lambda_{0}(S) \notin \sigma_{\ess}(S)$, $M_{1}$ being closed.
We know from Corollary \ref{glob defined V} that there exists $V \in C^{\infty}(M)$ with $p_{1*}H = - \grad \ln V$, such that for any section $s \colon U \subset M \to M_{2}$, the volume elements of the induced metrics on $G_{0}$ (and on $G$) satisfy
\[
d {\vol}_{g_{s(y)}} = V(y) d {\vol}_{g},
\]
where $g$ is a fixed left-invariant metric on $G$.

Observe that there exists a positive $\varphi \in C^{\infty}(M_{1})$ with $S \varphi = 0$, from \cite[Proposition 3.7]{Mine2} and the fact that $\lambda_{0}(S) = 0 \notin \sigma_{\ess}(S)$. Then $\mathcal{R}_{S}(\varphi) = 0$, which together with formula (\ref{special Rayleigh}), gives that $p_{*}H = - 2 \grad \ln \varphi$. It is now clear that $V$ is a constant multiple of the lift $\tilde{\varphi}^{2}$ of $\varphi^{2}$ on $M$, and in particular, $G/G_{0}$-invariant.

Given $z \in M_{2}$ and $x \in G$, using formula (\ref{volume eq}), the definition and the $G/G_{0}$-invariance of $V$, we compute
\[
R_{x}^{*} (d {\vol}_{g_{z}}) = d {\vol}_{g_{xz}} = V(p_{1}(xz))  d {\vol}_{g} = V([x]p_{1}(z))  d {\vol}_{g} = V(p_{1}(z))  d {\vol}_{g} =  d {\vol}_{g_{z}},
\]
where $[x]$ stands for the class of $x$ in $G/G_{0}$. Therefore, $g_{z}$ is a left-invariant metric on $G$ whose volume element is right-invariant, which means that $G$ is unimodular.
\qed\medskip

We end this section with a class of examples demonstrating that the assumption $\lambda_{0}(S) \notin \sigma_{\ess}(S)$ in Theorem \ref{submersion group thm}(iii) cannot be dropped, even if the manifolds are complete and the fibers are minimal.

\begin{example}
Let $G$ be the simply connected Lie group with Lie algebra spanned by two vectors $X,Y$ such that $[X,Y] = Y$. Given $c>0$, define the left-invariant metric $g_{c}$ on $G$ by $g_{c}(X,X) = c^{-1}$, $g_{c}(X,Y) = 0$ and $g_{c}(Y,Y) = c$. It is obvious that
\[
\langle \nabla_{X_{1}} X_{2} , X_{3} \rangle = \frac{1}{2} \left( \langle [X_{1},X_{2}] , X_{3} \rangle - \langle [X_{2},X_{3}] , X_{1} \rangle + \langle [X_{3},X_{1}] , X_{2} \rangle \right)
\]
for any left-invariant vector fields $X_{1},X_{2},X_{3}$ on $G$, where the inner products are with respect to $g_{c}$ and $\nabla$ stands for the Levi-Civita connection of $g_{c}$. From this, it is easy to see that $(G,g_{c})$ has constant sectional curvature $-c$. Thus, $(G,g_{c})$ is isometric to the 2-dimensional space form of sectional curvature $-c$, and in particular, the bottom of its spectrum is given by
\begin{equation}\label{sectional}
\lambda_{0}(G,g_{c}) = \frac{c^{2}}{4}.
\end{equation}
Bearing in mind that $G$ is solvable, observe that $G$ is not unimodular, from Theorem \ref{unim and amen}.

Let $M$ be a Riemannian manifold with $\lambda_{0}(M) \in \sigma_{\ess}(M)$. For a positive function $\psi \in C^{\infty}(M)$, endow the product manifold $M_{2} := M \times G$ with the Riemannian metric $g(y,x) = g_{M}(y) \times g_{\psi(y)}(x)$. It is evident that $G$ acts smoothly, freely and properly via isometries on $M_{2}$ and the Riemannian submersion arising from this action is the projection to the first factor $p \colon M_{2} \to M$. It is noteworthy that $p$ has minimal fibers, since the volume element of $g_{c}$ does not depend on $c$. Hence, the operator $S$ defined as in (\ref{operator}) coincides with the Laplacian on $M$.

By Proposition \ref{bottom essential spectrum}, there exists a sequence $(f_{n})_{n \in \mathbb{N}} \subset C^{\infty}_{c}(M) \smallsetminus \{0\}$ such that $\mathcal{R}(f_{n}) \rightarrow \lambda_{0}^{\ess}(M) = \lambda_{0}(M)$ and $\supp f_{n} \subset U_{n}$ for some precompact, open domains $U_{n}$ with $\bar{U}_{n}$ pairwise disjoint. Clearly, we may choose a positive $\psi \in C^{\infty}(M)$ with $\psi= c_{n} < 1/n$ in $U_{n}$ for any $n \in \mathbb{N}$. Then $p^{-1}(U_{n})$ is isometric to the Riemannian product $U_{n} \times G$, where $G$ is endowed with $g_{c_{n}}$. In view of Proposition \ref{bottom} and formula (\ref{sectional}), it follows that for any $n \in \mathbb{N}$ there exists $h_{n} \in C^{\infty}_{c}(G) \smallsetminus \{0\}$ with $\mathcal{R}_{g_{c_{n}}}(h_{n}) < 1/(4n^{2})$. Setting $\tilde{h}_{n}(y,x) = h_{n}(x)$ and $\tilde{f}_{n}(y,x) = f_{n}(y)$, we have that $\tilde{h}_{n} \tilde{f}_{n} \in C^{\infty}_{c}(M_{2})$ and a 
straightforward calculation implies that
\[
\mathcal{R}(\tilde{h}_{n} \tilde{f}_{n}) = \mathcal{R}_{g_{c_{n}}}(h_{n}) + \mathcal{R}(f_{n}) \rightarrow \lambda_{0}(M),
\]
as $n \rightarrow +\infty$. From this, together with Theorem \ref{basic mean curv thm} and Proposition \ref{bottom}, we deduce that $\lambda_{0}(M_{2}) = \lambda_{0}(M) = \lambda_{0}(S)$, while $G$ is not unimodular.
\end{example}

\section{Bottom of spectrum of Lie groups}\label{Lie group sec}

In this section, we discuss some applications of our results to Lie groups. We begin by establishing Theorem \ref{Lie group thm}. \medskip

\noindent\emph{Proof of Theorem \ref{Lie group thm}:} Apparently, the projection $p \colon G \to G/N$ is the Riemannian submersion arising from the (left) action of $N$ on $G$, and the fiber over $p(z)$ is $F_{p(z)} = Nz = zN$ for any $z \in G$, $N$ being normal. Since multiplication $L_{x}$ from the left with an element $x \in G$ maps isometrically $F_{p(z)}$ to $F_{p(xz)}$ for any $z \in G$, it is evident that the mean curvature $H$ of the fibers is left-invariant, and so is $p_{*}H$ on $G/N$. Then the operator $S$ on $G/N$ defined as in (\ref{operator}), is of the form $S = \Delta + c$ for some $c \in \mathbb{R}$, and the bottom of its spectrum is $\lambda_{0}(S) = \lambda_{0}(G/N) + c$.

To determine this constant, let $\{X_{i}\}_{i=1}^{m}$ be an orthonormal basis of $T_{e}G$ with $\{X_{i}\}_{i=1}^{k}$ spanning $T_{e}N$. Considering the left-invariant extension of $X_{i}$ (also denoted by $X_{i}$), it is easily checked that
\begin{eqnarray}\label{mean curv}
\| H \|^{2} &=& \sum_{i=1}^{k} \langle \nabla_{X_{i}} X_{i} ,H \rangle = - \sum_{i=1}^{k} \langle \nabla_{X_{i}} H ,X_{i} \rangle 
=  - \sum_{i=1}^{m} \langle \nabla_{X_{i}} H ,X_{i} \rangle + \sum_{i=k+1}^{m} \langle \nabla_{X_{i}} H ,X_{i} \rangle \nonumber \\
&=&  \sum_{i=1}^{m} \langle [ H,X_{i} ] ,X_{i} \rangle + \sum_{i=k+1}^{m} \langle \nabla_{p_{*}X_{i}} p_{*}H ,p_{*}X_{i} \rangle = \tr (\ad H) + \diver p_{*}H,
\end{eqnarray}
and the operator $S$ is written as
\[
S = \Delta - \frac{1}{4} \| H \|^{2} + \frac{1}{2} \tr (\ad H).
\]

The first statement now follows from Theorem \ref{basic mean curv thm}, after noticing that $\lambda_{0}(F_{y}) = \lambda_{0}(N)$ for any $y \in G/N$, $F_{y}$ being isometric to $N$. If $N$ is unimodular and amenable, then Theorem \ref{submersion group thm} establishes the asserted equality. Conversely, as a consequence of Theorem \ref{basic mean curv thm}, if
\[
\lambda_{0}(G) = \lambda_{0}(G/N) -\frac{1}{4} \| H \|^{2} + \frac{1}{2} \tr (\ad H),
\]
then the infimum of $\lambda_{0}(F_{y})$ with $y \in G/K$, is zero. Then $\lambda_{0}(N) = 0$, since $F_{y}$ is isometric to $N$ (endowed with the induced left-invariant metric from $G$), and Theorem \ref{unim and amen} yields that $N$ is unimodular and amenable. \qed \medskip

It is worth to point out that in the above setting, the assumption $\lambda_{0}(S) \notin \sigma_{\ess}(S)$ involved in Theorem \ref{submersion group thm}(iii) is not satisfied in general. Indeed, if $G/N$ is non-compact, then $\sigma(S) = \sigma_{\ess}(S)$, $S$ being invariant under multiplication from the left with elements of $G/N$ (cf. for instance \cite[Theorem 5.2]{Mine}). However, the conclusion of Theorem \ref{submersion group thm}(iii) holds because the fibers are isometric. 

\begin{corollary}
Let $G$ be a connected, unimodular and amenable Lie group endowed with a left-invariant metric and $N$ be a closed (as a subset), connected, normal subgroup of $G$ with mean curvature $H$. Then
\[
\lambda_{0}(G/N) = \frac{1}{4} \| H \|^{2}.
\]
In particular, $G/N$ is also unimodular (and amenable) if and only if $N$ is minimal.
\end{corollary}

\begin{proof}
Since $G$ is unimodular, we obtain from Lemma \ref{unimodular} that $\tr (\ad H) = 0$ and that $N$ is also unimodular.
According to Lemma \ref{short exact}, since $G$ is amenable, so are $N$ and $G/N$. The proof is completed by Theorems \ref{Lie group thm} and \ref{unim and amen}.\qed
\end{proof}\medskip

Recall that, in general, the quotient of a unimodular and amenable Lie group does not have to be unimodular. The next example demontrates this fact.

\begin{example}
Let $G$ be the simply connected, solvable Lie group with Lie algebra $\mathfrak{g}$ generated by $X,Y,Z$ satisfying $[X,Y] = Y$, $[X,Z] = - Z$ and $[Y,Z] = 0$. It is obvious that $\tr (\ad X^{\prime}) = 0$ for any $X^{\prime} \in \mathfrak{g}$, and we deduce from Lemma \ref{unimodular} that $G$ is unimodular. Let $N$ be the closed (as a subset), connected, normal subgroup of $G$ whose Lie algebra is the ideal generated by $Z$. Denoting by $p \colon G \to G/N$ the projection, it is elementary to verify that $\tr (\ad p_{*} X) = 1$. We conclude from Lemma \ref{unimodular} that $G/N$ is not unimodular, while $G$ is unimodular and amenable.
\end{example}

Before proceeding to the proof of Corollary \ref{bottom amen}, we need some auxiliary results. The next proposition provides a standard way of estimating the Cheeger constant of a Riemannian manifold.

\begin{proposition}\label{vector field}
	Let $X$ be a smooth vector field on a Riemannian manifold $M$ with $\| X \| \leq 1$ and $\diver X \geq c$ for some $c \in \mathbb{R}$. Then the Cheeger constant of $M$ is bounded by $h(M) \geq c$.
\end{proposition}

\begin{proof}
Using the divergence formula, for any smoothly bounded, compact domain $K$ of $M$, we compute
\[
c |K| \leq \int_{K} \diver X = \int_{\partial K} \langle X , \nu \rangle \leq | \partial K|,
\]
where $\nu$ is the outward pointing normal to $\partial K$. \qed
\end{proof}

\begin{corollary}\label{cheeger estimate}
Let $G$ be a connected Lie group endowed with a left-invariant metric. Then the Cheeger constant of $G$ satisfies
\[
h(G) \geq \max_{X \in \mathfrak{g}, \| X \| =1} \tr (\ad X).
\]
\end{corollary}

\begin{proof}
A straightforward calculation shows that $\tr (\ad X) = - \diver X$ for any $X \in \mathfrak{g}$, and the assertion is a consequence of Proposition \ref{vector field}. \qed
\end{proof}

\begin{proposition}\label{aux}
Let $G$ be a connected, amenable Lie group endowed with a left-invariant metric. Suppose that its radical $S$ is not abelian and denote by $H$ the mean curvature (in $G$) of the commutator subgroup $[S,S]$. Then
\[
\lambda_{0}(G) = \frac{1}{4} \| H \|^{2} = \frac{1}{4} \tr (\ad H).
\]
\end{proposition}

\begin{proof}
Consider the universal covering $q \colon \tilde{S} \to S$. Since $\tilde{S}$ is simply connected and solvable, it is known that its commutator subgroup $[\tilde{S},\tilde{S}]$ is closed (as a subset of $\tilde{S}$) and nilpotent (cf. for instance \cite[Proposition 1.6]{Hoke} and the references therein). This yields that the commutator subgroup $N := [S,S] = q([\tilde{S},\tilde{S}])$ is a connected, closed (as a subset), normal and nilpotent subgroup of $G$. Since connected, nilpotent groups are unimodular and amenable, Theorem \ref{unim and amen} gives that
\[
\lambda_{0}(G) = \lambda_{0}(G/N) - \frac{1}{4} \| H \|^{2} + \frac{1}{2} \tr (\ad H).
\]
	
Bearing in mind that $G$ is a compact extension of $S$, it is evident that $G/N$ is a compact extension of the abelian group $S/N$. In particular, $G/N$ is unimodular and amenable, and hence, $\lambda_{0}(G/N) = 0$, from Theorem \ref{unim and amen}. Let $\{X_{i}\}_{i=1}^{m}$ be an orthonormal basis of $\mathfrak{g}$ with $\{X_{i}\}_{i=1}^{k}$ spanning the Lie algebra of $N$. Then formula (\ref{mean curv}) yields that
\[
\| H \|^{2} = \tr (\ad H) - \tr (\ad p_{*}H).
\]
We derive from Lemma \ref{unimodular} that $\tr (\ad p_{*}H) = 0$, $G/N$ being unimodular, as we wished.
\qed
\end{proof}\medskip

\noindent\emph{Proof of Corollary \ref{bottom amen}:} 
If $G$ is unimodular, then the statement follows from Lemma \ref{unimodular}, Theorem \ref{unim and amen} and the Cheeger inequality. Suppose now that $G$ is not unimodular and observe that $S$ is not abelian, since $G$ is a compact extension of $S$. It follows from Theorem \ref{unim and amen} that $\lambda_{0}(G) > 0$, and thus, the mean curvature (in $G$) $H$ of the commutator subgroup $N:=[S,S]$ of the radical $S$ of $G$ is non-zero, from Proposition \ref{aux}. In view of Corollary \ref{cheeger estimate}, Proposition \ref{aux} and the Cheeger inequality, we conclude that
\[
\frac{1}{4} h(G)^{2} \geq \frac{1}{4} (\tr (\ad H_{0}))^{2} = \frac{1}{4} \tr (\ad H) = \lambda_{0}(G) \geq \frac{1}{4} h(G)^{2},
\]
where $H_{0} := \| H \|^{-1} H$. \qed \medskip

According to \cite{MR638814}, if the Cheeger constant coincides with the exponential volume growth, then the equality holds in the Cheeger inequality. However, this fails in Corollary \ref{bottom amen}, since there exist unimodular and amenable Lie groups of exponential volume growth (cf. \cite[p. 1525]{MR2053361} and the references therein).

\noindent Max Planck Institute for Mathematics \\
Vivatsgasse 7, 53111, Bonn \\
E-mail address: polymerp@mpim-bonn.mpg.de

\end{document}